\documentclass[11pt]{article}
\usepackage{listings}
\usepackage{pdflscape}
\usepackage{amsfonts}
\usepackage{epsfig}
\usepackage{lscape}
\usepackage{longtable}
\usepackage{amsmath,amssymb,amsthm}
\usepackage{graphicx}
\usepackage{subfigure}
\usepackage{color}
\usepackage{placeins}
\usepackage{url}
\usepackage{cases}
\usepackage{longtable}
\usepackage{supertabular}
\usepackage{multirow}
\newcommand{\tabincell}[2]{\begin{tabular}{@{}#1@{}}#2\end{tabular}}
\usepackage{amsmath}
\allowdisplaybreaks[4]
\usepackage{cite}
\usepackage{algorithmic}
\usepackage{algorithm,float}
\usepackage{booktabs}
\usepackage{enumitem}



\newtheorem{theorem}{Theorem}

\newtheorem{remark}{Remark}

\newtheorem{proposition}{Proposition}

\def\norm#1{\|#1\|}

\begin{document}

\title{\bf Adaptive Sieving with PPDNA: Generating Solution Paths of Exclusive Lasso Models\footnotemark[1]}
\author{Meixia Lin\footnotemark[2], \quad Yancheng Yuan\footnotemark[3], \quad Defeng Sun\footnotemark[4], \quad Kim-Chuan Toh\footnotemark[5]}
\date{\today}
\maketitle

\renewcommand{\thefootnote}{\fnsymbol{footnote}}
\footnotetext[1]{{\bf Funding:} Yancheng Yuan is supported by the National Research Foundation, Singapore under its International Research Centres in Singapore Funding Initiative, Defeng Sun is supported in part by the Hong Kong Research Grant Council grant PolyU 153014/18P and Kim-Chuan Toh by ARF grant R146-000-257-112 of the Ministry of Education of Singapore.}
\footnotetext[2]{Department of Mathematics, National University of Singapore, 10 Lower Kent Ridge Road, Singapore ({\tt lin\_meixia@u.nus.edu}).}
\footnotetext[3]{School of Computing, National University of Singapore, 13 Computing Drive, Singapore ({\tt yanchengyuanmath@gmail.com}).}
\footnotetext[4]{Department of Applied Mathematics, The Hong Kong Polytechnic University, Hung Hom, Hong Kong ({\tt defeng.sun@polyu.edu.hk}).}
\footnotetext[5]{Department of Mathematics and Institute of Operations Research and Analytics, National University of Singapore, 10 Lower Kent Ridge Road, Singapore ({\tt mattohkc@nus.edu.sg}).}
\renewcommand{\thefootnote}{\arabic{footnote}}

\begin{abstract}
   The exclusive lasso (also known as elitist lasso) regularization has become popular recently due to its superior performance on structured sparsity. Its complex nature poses difficulties for the computation of high-dimensional machine learning models involving such a regularizer. In this paper, we propose an adaptive sieving (AS) strategy for generating solution paths of machine learning models with the exclusive lasso regularizer, wherein a sequence of reduced problems with much smaller sizes need to be solved. In order to solve these reduced problems, we propose a highly efficient dual Newton method based proximal point algorithm (PPDNA). As important ingredients, we systematically study the proximal mapping of the weighted exclusive lasso regularizer and the corresponding generalized Jacobian. These results also make popular first-order algorithms for solving exclusive lasso models practical. Various numerical experiments  for the exclusive lasso models have demonstrated the effectiveness of the AS strategy for generating solution paths and the superior performance of the PPDNA.
\end{abstract}

\medskip
\noindent
{\bf Keywords:} Adaptive sieving, dual Newton method, proximal point algorithm, exclusive lasso
\\[5pt]
{\bf AMS subject classification:} 90C06, 90C25, 90C90

\section{Introduction}
\label{sec:intro}
For a given feature matrix $A=[a_1,a_2,\cdots,a_n] \in \mathbb{R}^{m \times n}$, we are interested in the machine learning models of the form:
\begin{align}
\min_{x \in \mathbb{R}^n} \ \Big\{ h(Ax)-\langle c,x\rangle + \lambda p(x)\Big\},\label{eq: general_ml_model}
\end{align}
where $c \in \mathbb{R}^{n}$ is a given vector, $h: \mathbb{R}^m \to \mathbb{R}$ is a convex twice continuously differentiable function, $p: \mathbb{R}^n \to (-\infty,+\infty]$, a closed and proper convex function, is a regularizer which usually induces structured sparsity for the solutions to the model, $\lambda > 0$ is a hyper-parameter which controls the trade-off between the loss and sparsity level of the solutions.

Many regularizers have been proposed to enforce sparsity with desirable structure in the predictors learned by machine learning models. For example, the lasso model \cite{tibshirani1996regression} can induce sparsity in the predictors but without structured patterns, and the group lasso model \cite{yuan2006model} can induce inter-group level sparsity. In some applications, intra-group level sparsity is desirable, which means that not only features from different groups, but also features in a seemingly cohesive group are competing to survive. One application comes from performing portfolio selections both across and within sectors in order to diversify the risk across different sectors. To achieve this intra-group sparsity, the exclusive lasso regularizer was proposed in \cite{zhou2010exclusive} (also named as elitist lasso \cite{kowalski2009sparse}), originally for multi-task learning. Since then, it has also been widely used in other applications later such as image processing \cite{zhang2016robust}, sparse feature clustering \cite{yamada2017localized} and NMR spectroscopy \cite{campbell2017within}. Let $w\in \mathbb{R}_{++}^n$ be a weight vector and $\mathcal{G} := \{g_1,\cdots,g_l\}$ be an index partition of the features such that $\bigcup_{j=1}^l g_j = \{1, 2, \dots, n\}$ and $g_j \bigcap g_k = \emptyset$ for any $j\neq k$. The corresponding  weighted exclusive lasso regularizer is defined as
\begin{align}
\Delta^{\mathcal{G},w}(x) :=\sum_{j=1}^l \|w_{g_j} \circ x_{g_j}\|_1^2, \quad \forall\, x \in \mathbb{R}^n,\label{eq: exclusive-lasso-regularizer}
\end{align}
where ``$\circ$'' denotes the Hadamard product, and $x_{g_j}$ denotes the sub-vector extracted from $x$ based on the index set $g_j$. Naturally, when solving exclusive lasso models, we can expect that each $x_{g_j}$ has nonzero coordinates under mild conditions, which means that every group has representatives.

One of the key tasks in applying machine learning models successfully in practice is in selecting suitable values for the hyper-parameters, in particular, the value of $\lambda$ in \eqref{eq: general_ml_model}. The model selection procedure \cite{ding2018model} usually requires one to solve the model \eqref{eq: general_ml_model} for a sequence of $\lambda$. To reduce the computation time of obtaining a solution path, especially for high-dimensional cases, various feature screening rules, which attempt to drop some inactive features based on prior analysis, have been proposed. Tibshirani et al. \cite{tibshirani2012strong} proposed a strong screening rule (SSR) based on the ``unit-slope" bound assumption for the lasso model and generalized it to some other cases. Although the SSR performs well in practice, it may screen out some active features by mistake. To avoid this, safe screening rules have been studied. The first safe rule was proposed by Ghaoui et al. \cite{ghaoui2010safe} for lasso models. Later on, Wang et al. \cite{wang2013lasso, wang2015lasso} proposed a dual prototype projection based screening rule (DPP) and an enhanced version (EDPP) for lasso and group lasso models via carefully analyzing the geometry of corresponding dual problems. Other safe screening rules, like Sphere test \cite{xiang2016screening}, have been proposed via different strategies for estimating a compact region of the optimal solution to the dual problem. Recently, Zeng et al. \cite{zeng2017efficient} combined the SSR and EDPP to propose a hybrid safe-strong screening rule, which is implemented in an R package {\tt biglasso} \cite{zeng2017biglasso}. Unfortunately, all these screening rules are difficult to be generalized to the exclusive lasso models. In this paper, we propose an adaptive sieving (AS) strategy for generating solution paths of machine learning models of the form \eqref{eq: general_ml_model}, including those with the exclusive lasso regularizer. The proposed AS strategy is directly based on the Karush-Kuhn-Tucker (KKT) condition, and it does not depend heavily on the specific form of the regularizer, as long as the subdifferential of the regularizer $p(\cdot)$ can be characterized explicitly. Numerical experiments demonstrate that the AS strategy is very effective in reducing the problem dimensions.
	
In order to bring the AS strategy for exclusive lasso models into full play, we need an efficient algorithm to solve the involved reduced problems to satisfactory level of accuracy. As we will see in Section \ref{sec: screening}, the reduced problems are in the same form of the original model, but with smaller sizes. Unfortunately, current state-of-the-art algorithms for solving exclusive lasso models, such as the iterative least squares algorithm (ILSA) \cite{kong2014exclusive,yamada2017localized}, the coordinate descent method \cite{campbell2017within}, are very time-consuming to obtain a solution with high accuracy, even for problems with medium sizes. In addition, popular first-order algorithms, such as the accelerated proximal gradient method (APG) \cite{zhang2016robust}, FISTA \cite{beck2009fast} and alternating direction method of multipliers (ADMM) \cite{eckstein1992douglas,glowinski1975approximation}, have not been widely used to solve exclusive lasso models. The main reason may lie in the fact that the proximal mapping of the exclusive lasso regularizer, which is the key ingredient for the efficient implementation of algorithms mentioned above, has not been systematically studied yet. Yoon and Hwang provided a closed-form solution for the proximal mapping of $\Delta^{\mathcal{G},w}(\cdot)$ in \cite{yoon2017combined} but unfortunately it is incorrect. Kowalski mentioned the proximal mapping in \cite{kowalski2009sparse}, but the derivation contains some errors, and this result is not known to most  researchers in the optimization and machine learning communities. In this paper, we systematically study the exclusive lasso regularizer, and provide a closed-form solution to the proximal mapping of the general weighted exclusive lasso regularizer \eqref{eq: exclusive-lasso-regularizer}. Such a closed-form solution is important for the practical efficiency of some algorithmic frameworks, such as APG and ADMM, for solving exclusive lasso models. However, as we shall see in the numerical experiments, even with the closed-form solution of the proximal mapping, first-order algorithms, such as APG and ADMM, are not efficient enough. To overcome this computational challenge, we design a highly efficient second-order type algorithm, the dual Newton method based proximal point algorithm (PPDNA), to solve exclusive lasso models. As a key ingredient of the PPDNA, we carefully derive the generalized Jacobian of the proximal mapping of the weighted exclusive lasso regularizer. Numerical results demonstrate the superior performance of the PPDNA against ILSA, APG and ADMM for solving exclusive lasso models.

We summarize our main contributions in this paper as follows.
\begin{itemize}[noitemsep,topsep=0pt]
	\item[1] We propose a simple yet powerful adaptive sieving strategy for generating solution paths of general machine learning models, including those with the exclusive lasso regularizer. To the best of our knowledge, this is the first practical screening rule for exclusive lasso models. Our AS strategy allows each reduced problem to be solved inexactly, and is proved to converge in finite iterations. Extensive numerical experiments are also conducted to demonstrate the power of the AS strategy for obtaining solution paths of exclusive lasso models.
	\item[2] We develop a highly efficient dual Newton method based proximal point algorithm to solve the reduced problems in the AS strategy for solving the exclusive lasso model. We prove that the error bound condition holds for commonly used exclusive lasso models, which guarantees the superlinear convergence of the PPDNA. Numerical experiments are also performed to demonstrate the superior performance of the PPDNA against other state-of-the-art algorithms for solving the exclusive lasso models.
	\item[3] As key ingredients of the PPDNA, we systematically study the proximal mapping of the weighted exclusive lasso regularizer, and the corresponding generalized Jacobian. These results are also critical in computing the key projection step of various first-order algorithms for solving the exclusive lasso models.
\end{itemize}

The rest of the paper is organized as follows. In Section \ref{sec: screening}, we propose an adaptive sieving strategy for solution paths of a general machine learning model. In order to efficiently solve the reduced problems involved in the AS strategy for exclusive lasso models, we design a dual Newton method based proximal point algorithm in Section \ref{sec:pppa}. In Section \ref{sec:proxJacobian}, we systematically study the weighted exclusive lasso regularizer, through providing the closed-form solution to the proximal mapping and its generalized Jacobian. Numerical experiments on both synthetic data and real data are presented in Section \ref{sec:numerical}, which demonstrate the power of the AS strategy with the PPDNA for obtaining solution paths of exclusive lasso models. In the end, we conclude the paper.

\vspace{0.2cm}
\noindent\textbf{Notations and preliminaries:} Denote $\mathbb{B}_{\infty}$ ($\mathbb{B}_{2}$) as the infinity norm ($2$-norm) unit ball. For any $z\in \mathbb{R}$, ${\rm sign}(z)$ denotes the sign function of $z$, and $z^{+}:=\max\{z,0\}$, $z^{-}:=\min\{z,0\}$. We use ``${\rm Diag}(x)$" to denote the diagonal matrix whose diagonal is given by the vector $x$, and use ``${\rm Diag}(X_1,\cdots,X_n)$" to denote the block diagonal matrix whose $i$-th diagonal block is the matrix $X_i$, $i=1,\cdots,n$. For any self-adjoint positive semidefinite linear operator ${\cal M}:\mathbb{R}^n\rightarrow \mathbb{R}^n$, we define $\langle x,x'\rangle_{\cal M}:=\langle x,{\cal M}x'\rangle$, and $\|x\|_{\cal M}:=\sqrt{\langle x,x\rangle_{\cal M}}$ for all $x,x'\in \mathbb{R}^n$. For a given subset ${\cal C}$ of $\mathbb{R}^n$, we denote the weighted distance of $x\in \mathbb{R}^n$ to ${\cal C}$ as ${\rm dist}_{\cal M}(x,{\cal C}):=\inf_{x'\in {\cal C}}\|x-x'\|_{\cal M}$. The largest eigenvalue of ${\cal M}$ is denoted as $\lambda_{\max}({\cal M})$.

Let  $q:\mathbb{R}^n\rightarrow (-\infty,\infty]$ be a  closed and proper convex function. The conjugate function of $q$ is defined as $q^*(z):=\sup_{x\in\mathbb{R}^n}\{\langle x,z\rangle-q(x)\}$. The Moreau envelope of $q$ at $x$ is defined by
\begin{align*}
{\rm E}_q(x):=\min_{y\in \mathbb{R}^n}\Big\{ q(y)+\frac{1}{2}\|y-x\|^2\Big\},
\end{align*}
and the corresponding proximal mapping ${\rm Prox}_q(x)$ is defined as the unique optimal solution of the above problem. It is known that for any $x\in \mathbb{R}^n$, $\nabla {\rm E}_q(x)=x-{\rm Prox}_q(x)$, and ${\rm Prox}_q(\cdot)$ is Lipschitz continuous with modulus $1$ \cite{moreau1965proximite,rockafellar1976monotone}.

In order to study the weighted exclusive lasso regularizer $\Delta^{\mathcal{G},w}(\cdot)$ defined in \eqref{eq: exclusive-lasso-regularizer}, we use the following notations. For $j=1\cdots,l$, we define the linear mapping ${\cal P}_j:\mathbb{R}^n \rightarrow \mathbb{R}^{|g_j|}$ as ${\cal P}_j x=x_{g_j}$ for all $x\in \mathbb{R}^n$, and ${\cal P}=[{\cal P}_1;\cdots;{\cal P}_l]$. Let $n_j=\sum_{k=1}^j |g_k|$ and $n_0=0$. Denote $x^{(j)}$ as the sub-vector extracted from $x$ based on the index set $\{n_{j-1}+1,n_{j-1}+2,\cdots,n_{j}\}$ for $j=1,\cdots,l$. According to these notations, we have
\begin{align}
\Delta^{\mathcal{G},w}(x) =\sum_{j=1}^l\|({\cal P}w)^{(j)}\circ ({\cal P}x)^{(j)}\|_1^2, \quad \forall\, x \in \mathbb{R}^n.\label{eq: reformulation_exclusive}
\end{align}

\section{An adaptive sieving strategy}
\label{sec: screening}
In this section, we propose an adaptive sieving strategy, and apply it to obtain solution paths of general machine learning models, including those with the exclusive lasso regularizer. From now on, we consider a more general form of \eqref{eq: general_ml_model}, given by
\begin{align}
\label{eq: lasso_model}
\min_{x\in \mathbb{R}^n}\ \displaystyle \Big\{\Phi(x) + \lambda p(x)\Big\},\tag{$\mbox{P}_{\lambda}$}
\end{align}
where $\Phi:\mathbb{R}^n\rightarrow \mathbb{R}$ is convex twice continuously differentiable. Denote the optimal solution set of \eqref{eq: lasso_model} as $\Omega_{\lambda}$, and the proximal residual function $R_{\lambda}:\mathbb{R}^n\rightarrow \mathbb{R}^n$ associated with \eqref{eq: lasso_model} as
\begin{align}
R_{\lambda}(x):=x-{\rm Prox}_{\lambda p}(x-\nabla \Phi(x)),\quad \forall x\in \mathbb{R}^n.\label{eq:residual}
\end{align}
The KKT condition of \eqref{eq: lasso_model} implies that $\bar{x}\in \Omega_{\lambda}$ if and only if $R_{\lambda}(\bar{x})=0$. In this paper, we assume that for any $\lambda > 0$, the solution set $\Omega_{\lambda}$ to \eqref{eq: lasso_model} is nonempty and compact. For many popular machine learning models, this assumption is satisfied automatically, as discussed in \cite[Section 2.1]{zhou2017unified}.

To reduce the computation time of obtaining solution paths of various machine learning models, different kinds of feature screening rules have been proposed. But they are limited by the fact that they are highly dependent on the positive homogeneity of the regularizers and they implicitly require the reduced problems to be solved exactly. We first briefly discuss the ideas behind two most popular screening rules, namely the strong screening rule (SSR) \cite{tibshirani2012strong} and the dual prototype projection based screening rule (DPP) \cite{wang2013lasso, wang2015lasso}. For convenience, we take the lasso model as an illustration, with $\Phi(x) = \frac{1}{2}\|Ax - b\|^2$ and $p(x) = \|x\|_1$ in \eqref{eq: lasso_model}. Let $x^{*}(\lambda)\in \Omega_{\lambda}$. The KKT condition implies that:
\begin{align*}
a_i^T\theta^*(\lambda) \in \left\{
\begin{array}{ll}
\{ \lambda \ {\rm sign}(x^*(\lambda))\} & \mbox{if} \; x^*(\lambda)_i \not= 0 \\
\left[-\lambda, \lambda\right] & \mbox{if} \; x^*(\lambda)_i = 0\\
\end{array}
\right.,
\end{align*}
where $\theta^*(\lambda)$ is the optimal solution of the associated dual problem:
\begin{align}
\label{eq: dual_problem_lasso}
\max_{\theta\in \mathbb{R}^m} \ \Big\{\frac{1}{2}\|b\|^2 - \frac{1}{2}\|\theta - b\|^2\; \mid \;|a_i^T\theta| \leq \lambda,\ i = 1, 2, \dots, n\Big\}.
\end{align}
The existing screening rules for the lasso model are based on the fact that $x^*(\lambda)_i = 0$ if $|a_i^T\theta^*(\lambda)| < \lambda$. The difference is how to estimate $a_i^T\theta^*(\lambda)$ based on an optimal solution $x^*(\tilde{\lambda})$ of $({\rm P}_{\tilde{\lambda}})$ for some $\tilde{\lambda} > \lambda$, without solving the dual of $({\rm P}_{\lambda})$. The SSR \cite{tibshirani2012strong} discards the $i$-th predictor if
\begin{align*}
|a_i^T\theta^*(\tilde{\lambda}) | \leq 2\lambda-\tilde{\lambda},
\end{align*}
by assuming the ``unit slope" bound condition:
\begin{align*}
|a_i^T\theta^*(\lambda_1) - a_i^T\theta^*(\lambda_2)| \leq |\lambda_1-\lambda_2|, \quad \forall \lambda_1, \lambda_2 > 0.
\end{align*}
Since this assumption may fail, the SSR may screen out some active features by mistake. Also, the SSR only works for consecutive hyper-parameters with a small gap, since it requires $\lambda > \frac{\tilde{\lambda}}{2}$. Wang et al. \cite{wang2013lasso, wang2015lasso} proposed the DPP by carefully analyzing the properties of the optimal solution to the dual problem \eqref{eq: dual_problem_lasso}. The key idea is, if we could estimate a region $\Theta_{\lambda}$ containing $\theta^*(\lambda)$, then
\begin{align*}
\sup_{\theta \in \Theta_{\lambda}} |a_i^T\theta| < \lambda \quad \Longrightarrow \quad x^*(\lambda)_i = 0.
\end{align*}
They estimate the region $\Theta_{\lambda}$ by realizing that the optimal solution of \eqref{eq: dual_problem_lasso} is the  projection onto a polytope. As we can see, a tighter estimation of $\Theta_{\lambda}$ will induce a better safe screening rule.
	
Unfortunately, these popular screening rules are difficult to be generalized to the exclusive lasso model. First, as we will derive later in \eqref{eq: subdiff_exclusive_2}, the subdifferential of the exclusive lasso regularizer is much more complicated than the lasso regularizer, whose subdifferential is separable. In order to apply these screening rules, we need to at least get a tight lower bound of $\|w_{g_j} \circ x^*(\lambda)_{g_j}\|_1$ for each feature group $g_j$, which is difficult. Second, since the exclusive lasso regularizer is not positively homogeneous, the optimal solution to the dual problem is not the projection onto some convex set \cite[Theorem 13.2]{rockafellar1970convex}.

To overcome the challenges just mentioned, we propose an adaptive sieving strategy, which does not depend on the specific form of the regularizer and thus can be applied to a general regularizer. More importantly, due to the adaptive nature of the AS strategy, it can sieve out a very large proportion of inactive features, as we shall see in the numerical experiments. In addition, the AS strategy also allows the involved reduced problems to be solved inexactly, provided the corresponding error tolerances can be analyzed explicitly.

\subsection{The adaptive sieving strategy}
In this section, we propose an adaptive sieving strategy, directly based on the KKT condition, for generating solution paths of general machine learning models, including those with the exclusive lasso regularizer. Here we suppose that the model contains a regularization term which induces sparsity in the solutions. The details of the AS strategy could be found in Algorithm \ref{alg:screening}.

\begin{algorithm}[H]\small
	\caption{Adaptive sieving strategy for solving {($\mbox{P}_{\lambda}$})} 
	\label{alg:screening}
	\begin{algorithmic}[1]
		\STATE \textbf{Input}: a sequence of hyper-parameter: $\lambda_0 > \lambda_1 > \dots > \lambda_k > 0$, and tolerance $\epsilon \geq 0$.
		\STATE \textbf{Output}: a solution path: $x^*(\lambda_0), x^*(\lambda_1), x^*(\lambda_2), \dots, x^*(\lambda_k)$.
		\STATE \textbf{Initialization}: for $\lambda_0 > 0$, solve
		\begin{align}
		x^{*}(\lambda_0) \in \underset{ x\in \mathbb{R}^n} {\arg\min} \  \Big\{\Phi(x)  + \lambda_0 p(x) - \langle \delta_0,x\rangle\Big\},\label{eq: lambda0_problem}
		\end{align}
		where $\delta_0\in \mathbb{R}^n$ is an error vector such that $\|\delta_0\|\leq \epsilon$. Let
		\begin{align*}
		I^*(\lambda_0) := \{j \;\mid\; x^*(\lambda_0)_j \neq  0,\ j=1,\cdots,n\}.
		\end{align*}
		\FOR{$i = 1, 2, \dots, k$}
		\STATE \textbf{1}. Let $I^{0}(\lambda_i) = I^*(\lambda_{i-1})$. Find
		\begin{align*}
		x^{0}(\lambda_i) \in \underset{ x\in \mathbb{R}^n} {\arg\min} \  \Big\{\Phi(x)  + \lambda_i p(x) - \langle \delta_i^0,x\rangle	\ \mid \ x_{\bar{I}^{0}(\lambda_i)}=0\Big\},
		\end{align*}
		where $\bar{I}^0(\lambda_i)$ denotes the complement of $I^0(\lambda_i)$ in $\{1,\ldots,n\}$, $\delta_i^0\in \mathbb{R}^n$ is an error vector such that $\|\delta_i^0\|\leq \epsilon/\sqrt{2}$, $(\delta_i^0)_{\bar{I}^{0}(\lambda_i)}=0$.
		
		\textbf{2}. Compute $R_{\lambda_i}(x^0(\lambda_i))$ and set $l=0$.
		\WHILE{$\|R_{\lambda_i}(x^l(\lambda_i))\|> \epsilon$}
		\STATE \textbf{3.1}. Create $J^{l+1}(\lambda_i)$:
		\begin{align}
		J^{l+1}(\lambda_i) = \Big\{ j\in \bar{I}^{l}(\lambda_i)\; \mid \; -\Big(  \nabla \Phi(x^l(\lambda_{i}))\Big)_j \notin \lambda_i \Big(\partial p(x^l(\lambda_{i}))+\frac{\epsilon}{\lambda_i \sqrt{2|\bar{I}^{l}(\lambda_i)|}}\mathbb{B}_{\infty}
		\Big)_j\Big\},\label{eq: create_J}
		\end{align}
		where $\bar{I}^l(\lambda_i)$ denotes the complement of $I^l(\lambda_i)$ in $\{1,\ldots,n\}$, and $\big({\cal C}\big)_j$ denotes the projection of the set ${\cal C}$ onto the $j$-th dimension. Then update $I^{l+1}(\lambda_i)$ as:		
		\begin{align*}
		I^{l+1}(\lambda_i) \leftarrow I^{l}(\lambda_i) \cup J^{l+1}(\lambda_i) .
		\end{align*}
		\STATE \textbf{3.2}. Solve the following constrained problem:
		\begin{align}
		x^{l+1}(\lambda_i) \in \underset{ x\in \mathbb{R}^n} {\arg\min} \  \Big\{\Phi(x)  + \lambda_i p(x) - \langle \delta_i^{l+1},x\rangle \ \mid \ x_{\bar{I}^{l+1}(\lambda_i)}=0\Big\},\label{eq: constrained}
		\end{align}
		where $\delta_i^{l+1}\in \mathbb{R}^n$ is an error vector such that $\|\delta_i^{l+1}\|\leq \epsilon/\sqrt{2}$, $(\delta_i^{l+1})_{\bar{I}^{l+1}(\lambda_i)}=0$.
		\STATE \textbf{3.3}: Compute $R_{\lambda_i}(x^{l+1}(\lambda_i))$ and set $l\leftarrow l+1$.
		\ENDWHILE
		\STATE \textbf{4}. Set $x^*(\lambda_i)=x^l(\lambda_{i})$ and
		$I^*(\lambda_i)=I^{l}(\lambda_i)$.
		\ENDFOR
	\end{algorithmic}
\end{algorithm}

Note that in Algorithm 1, the introduction of the error vectors $\delta_0$, $\{\delta_i^{l+1}\}$ in \eqref{eq: lambda0_problem} and \eqref{eq: constrained} implies  that the corresponding minimization problems could be solved inexactly. It is important to note that the vectors are not a priori given but they are the errors incurred when the original problems (with $\delta_0=0$ in \eqref{eq: lambda0_problem} and $\delta_i^{l+1}=0$ in \eqref{eq: constrained}) are solved inexactly. We take the problem \eqref{eq: lambda0_problem} as an example to explain how the error vector $\delta_0$ is obtained in the following proposition.
\begin{proposition}
	The updating rule of $x^*(\lambda_0)$ in \eqref{eq: lambda0_problem} can be interpreted as follows: compute
	\begin{align*}
	x^*(\lambda_0):={\rm Prox}_{\lambda_0 p}(\hat{x}-\nabla \Phi(\hat{x})),
	\end{align*}
	where $\hat{x}$ is an approximate solution to $(P_{\lambda_0})$ such that
	\begin{align}
	\|R_{\lambda_0}(\hat{x})+\nabla \Phi({\rm Prox}_{\lambda_0 p}(\hat{x}-\nabla \Phi(\hat{x})))-\nabla \Phi(\hat{x})\|\leq \epsilon/\sqrt{2}.\label{eq: control_delta0}
	\end{align}
	If $\nabla \Phi(\cdot)$ is Lipschitz continuous with modulus $L$, the condition \eqref{eq: control_delta0} could be achieved by
	\begin{align*}
	\|R_{\lambda_0}(\hat{x})\|\leq \frac{\epsilon}{\sqrt{2}(1+L)}.
	\end{align*}
\end{proposition}
\begin{proof}
	Let $\{x^{i}\}$ be a sequence that converges to a solution of the problem \eqref{eq: lambda0_problem} with $\delta_0=0$. For any $i= 1,2,\cdots$, define $\delta^{i}:=R_{\lambda_0}(x^{i})+\nabla \Phi({\rm Prox}_{\lambda_0 p}(x^{i}-\nabla \Phi(x^{i})))-\nabla \Phi(x^{i})$. By the continuous differentiability of $\Phi(\cdot)$ and \cite[Lemma 4.5]{mengyu2015inexact}, we know that $\lim_{i\rightarrow \infty}\|\delta^{i}\|=0$, which implies the existence of $\hat{x}$ in \eqref{eq: control_delta0}.
	
	Next we explain the reason why the updating rule of $x^*(\lambda_0)$ in \eqref{eq: lambda0_problem} could be interpreted as the one stated in the proposition. Since $R_{\lambda_0}(\hat{x})=\hat{x}-{\rm Prox}_{\lambda_0 p}(\hat{x}-\nabla \Phi(\hat{x}))=\hat{x}-x^*(\lambda_0)$, we have
	\begin{align*}
	R_{\lambda_0}(\hat{x})-\nabla \Phi(\hat{x})\in \lambda_0 \partial p(x^*(\lambda_0)).
	\end{align*}
	If we choose $\delta_0:=R_{\lambda_0}(\hat{x})+\nabla \Phi(x^*(\lambda_0))-\nabla \Phi(\hat{x})$, then we have
	\begin{align*}
	\delta_0 \in \nabla \Phi(x^*(\lambda_0))+\lambda_0 \partial p(x^*(\lambda_0)),
	\end{align*}
	which means that $x^*(\lambda_0)$ is the exact solution of the problem \eqref{eq: lambda0_problem} with the above $\delta_0$. Moreover,
	\begin{align*}
	\|\delta_0\|= \|R_{\lambda_0}(\hat{x})+\nabla \Phi(x^*(\lambda_0))-\nabla \Phi(\hat{x})\|\leq \epsilon/\sqrt{2}.
	\end{align*}
	The remaining conclusion follows naturally from the assumption of Lipschitz property of $\nabla\Phi(\cdot)$ and the triangle inequality.
\end{proof}

Now, we show the convergence properties of Algorithm \ref{alg:screening} in the following proposition, which imply that the algorithm will terminate after a finite number of iterations.

\begin{proposition}\label{prop: convergence_whileloop}
	For each $i\in \{1,\cdots,k\}$, the while loop in Algorithm \ref{alg:screening} will terminate after a finite number of iterations.
\end{proposition}
\begin{proof}
	Given $i\in \{1,\cdots,k\}$. We first prove that when $\|R_{\lambda_i}(x^l(\lambda_i))\|>\epsilon\geq 0$ for some $l\geq 0$, the index set $J^{l+1}(\lambda_i)$ defined in \eqref{eq: create_J} is nonempty. We prove this by contradiction. Suppose that $J^{l+1}(\lambda_i)=\emptyset$, which means
	\begin{align*}
	-\Big(\nabla \Phi(x^l(\lambda_{i}))\Big)_j \in \lambda_i \Big(\partial p(x^l(\lambda_{i})+\frac{\epsilon}{\lambda_i\sqrt{2|\bar{I}^{l}(\lambda_i)|}}\mathbb{B}_{\infty}\Big)_j,\quad \forall\; j\in \bar{I}^{l}(\lambda_i).
	\end{align*}
	That is, there exists a vector $\hat{\delta}_i^{l}\in \mathbb{R}^n$ with $(\hat{\delta}_i^{l})_{I^{l}(\lambda_i)}=0$, $\|\hat{\delta}_i^{l}\|_{\infty}\leq \frac{\epsilon}{\sqrt{2|\bar{I}^{l}(\lambda_i)|}}$, such that
	\begin{align}
	-\Big( \nabla \Phi(x^l(\lambda_{i}))-\hat{\delta}_i^{l}
	\Big)_j \in \lambda_i \Big(\partial p(x^l(\lambda_{i})\Big)_j,\quad \forall\; j\in
	\bar{I}^{l}(\lambda_i).
	\label{eq: J_empty}
	\end{align}
	Note that $x^{l}(\lambda_i)$ satisfies
	\begin{align*}
	x^{l}(\lambda_i) \in \underset{ x\in \mathbb{R}^n} {\arg\min} \  \Big\{\Phi(x)  + \lambda_i p(x) - \langle \delta_i^{l},x\rangle \ \mid \ x_{\bar{I}^{l}(\lambda_i)}=0\Big\},
	\end{align*}
	where $\delta_i^{l}\in \mathbb{R}^n$ is an error vector such that $\|\delta_i^{l}\|\leq \epsilon/\sqrt{2}$ and $(\delta_i^{l})_{\bar{I}^{l}(\lambda_i)}=0$. By the KKT condition of the above minimization problem, we know that there exists $y\in \mathbb{R}^{|\bar{I}^{l}(\lambda_i)|}$ such that
	\begin{align*}
	\left\{
	\begin{aligned}
	&0\in\Big(\nabla \Phi(x^l(\lambda_{i}))-\delta_i^{l}\Big)_{I^{l}(\lambda_i)}+\lambda_i \Big(\partial p(x^{l}(\lambda_i))\Big)_{I^{l}(\lambda_i)},\\
	&0\in \Big(\nabla \Phi(x^l(\lambda_{i}))-\delta_i^{l}\Big)_{\bar{I}^{l}(\lambda_i)}
	+\lambda_i \Big(\partial p(x^{l}(\lambda_i))\Big)_{\bar{I}^{l}(\lambda_i)}-y,\\
	& \Big( x^{l}(\lambda_i)\Big)_{\bar{I}^{l}(\lambda_i)}=0.
	\end{aligned}\right.
	\end{align*}
	Therefore, together with \eqref{eq: J_empty}, we have
	\begin{align*}
	- \nabla \Phi(x^l(\lambda_{i}))+\tilde{\delta}_i^{l}\in \lambda_i  \partial p(x^{l}(\lambda_i)),
	\end{align*}
	where $\tilde{\delta}_i^{l}\in \mathbb{R}^n$ is defined as $(\tilde{\delta}_i^{l})_{I^{l}(\lambda_i)}=(\delta_i^{l})_{I^{l}(\lambda_i)}$, $(\tilde{\delta}_i^{l})_{\bar{I}^{l}(\lambda_i)}=(\hat{\delta}_i^{l})_{\bar{I}^{l}(\lambda_i)}$, which means
	\begin{align*}
	x^{l}(\lambda_i)= {\rm Prox}_{\lambda_i p}(x^{l}(\lambda_i)- \nabla \Phi(x^l(\lambda_{i}))+\tilde{\delta}_i^{l}).
	\end{align*}
	As a result, it holds that
	\begin{align*}
	\|R_{\lambda_i}(x^{l}(\lambda_i))\|&=\|x^{l}(\lambda_i)- {\rm Prox}_{\lambda_i p}(x^{l}(\lambda_i)- \nabla \Phi(x^l(\lambda_{i})))\|\\
	&=\|{\rm Prox}_{\lambda_i p}(x^{l}(\lambda_i)- \nabla \Phi(x^l(\lambda_{i}))+\tilde{\delta}_i^{l})- {\rm Prox}_{\lambda_i p}(x^{l}(\lambda_i)- \nabla \Phi(x^l(\lambda_{i})))\|\leq \|\tilde{\delta}_i^{l}\|\leq \epsilon,
	\end{align*}
	which is a contradiction.
	
	Therefore, we have that for $l\geq 0$, $J^{l+1}(\lambda_i)\neq \emptyset$  as long as $R_{\lambda_i}(x^l(\lambda_i))>\epsilon$ . In other words, new indices will be added to the index set $I^{l+1}(\lambda_i)$ as long as the KKT residual has not achieved at the required accuracy. Since the total number of features $n$ is finite, the while loop in Algorithm \ref{alg:screening} will terminate after a finite number of iterations.
\end{proof}

Note that the KKT condition of \eqref{eq: lambda0_problem} implies that
\begin{align*}
x^*(\lambda_0)  = {\rm Prox}_{\lambda_0 p}(x^*(\lambda_0)- \nabla \Phi(x^*(\lambda_0))+\delta_0).
\end{align*}
By using the property that the  proximal mapping ${\rm{Prox}}_{\lambda_0 p}(\cdot)$ is Lipschitz continuous with modulus 1, we have the following estimation
\begin{align*}
\|R_{\lambda_0}(x^*(\lambda_0))\|&=\| {\rm Prox}_{\lambda_0 p}(x^*(\lambda_0)- \nabla \Phi(x^*(\lambda_0))+\delta_0)- {\rm Prox}_{\lambda_0 p}(x^*(\lambda_0)-\nabla \Phi(x^*(\lambda_0)))\|\leq \|\delta_0\|\leq \epsilon.
\end{align*}
Thus, together with Proposition \ref{prop: convergence_whileloop}, we obtain the following theorem regarding the convergence   properties of Algorithm \ref{alg:screening} directly.
\begin{theorem}\label{thm: convergence_screening}
	The solution path  $x^*(\lambda_0), x^*(\lambda_1),\dots, x^*(\lambda_k)$  generated by Algorithm \ref{alg:screening} are approximate optimal solutions to the problems $(\mbox{P}_{\lambda_0}), (\mbox{P}_{\lambda_1}),\cdots, (\mbox{P}_{\lambda_k})$, respectively, in the sense that \[\|R_{\lambda_i}(x^*(\lambda_i))\|\leq \epsilon,\quad  i=0,1,\cdots,k.\]
\end{theorem}

\subsection{Practical implementations}
As one can see, there are two key tasks involved in applying the proposed AS strategy for solving machine learning models:
\begin{itemize}
	\item[(1)] practical and efficient implementation of the construction of $J^{l+1}(\lambda_i)$ in \eqref{eq: create_J};
	\item[(2)] an efficient and robust algorithm for solving the problem \eqref{eq: lambda0_problem} and the constrained problem \eqref{eq: constrained} to the required accuracy.
\end{itemize}
Next, we discuss these two aspects in details.

\paragraph{Construction of $J^{l+1}(\lambda_i)$.} For the construction of $J^{l+1}(\lambda_i)$ in \eqref{eq: create_J}, we need to fully characterize $\partial p(x)$ for any $x\in \mathbb{R}^n$. Fortunately, the subdifferential of many popular regularizers, such as the lasso, group lasso and exclusive lasso regularizers, could be explicitly computed. We take the weighted exclusive lasso regularizer as an example. From the reformulation of the weighted exclusive lasso regularizer in \eqref{eq: reformulation_exclusive}, it could be seen that
\begin{align}
\partial \Delta^{\mathcal{G},w}(x)=\{{\cal P}^T u\in \mathbb{R}^n: \  u^{(j)}\in \partial q_j(({\cal P}x)^{(j)}),\ j = 1,\cdots, l\}, \label{eq: subdiff_exclusive}
\end{align}
where $q_j(\cdot):=\|({\cal P}w)^{(j)}\circ \cdot\|_1^2$ for $j=1,\cdots, l$. From the chain rule of subdifferential \cite[Theorem 2.3.9]{clarke1990optimization}, we know that for each $j$,
\begin{align*}
\partial q_j(v)= \{2\|({\cal P}w)^{(j)}\circ v\|_1\hat{v}: \ \hat{v} \in \partial \|({\cal P}w)^{(j)}\circ v\|_1\}.
\end{align*}
As we know that given any $\bar{n}\geq 1$ and $\beta\in \mathbb{R}^{\bar{n}}_{++}$,
\begin{align*}
\partial \|\beta\circ z\|_1={\rm SGN}_{\beta}(z),\quad \forall z\in \mathbb{R}^{\bar{n}},
\end{align*}
where
\begin{align}
{\rm SGN}_{\beta}(z):=\left\{u\in \mathbb{R}^{\bar{n}}:u_j\in \left\{\begin{aligned}
&\{\beta_j {\rm{sign}}(z_j)\} && \mbox{if}\ z_j \not= 0 \\
&[-\beta_j,\beta_j] && \mbox{if}\ z_j=0\\
\end{aligned}\right.,\ j = 1,\cdots,\bar{n}\right\}.\label{eq: define_SGN}
\end{align}
For simplicity, we denote ${\rm SGN}(z):={\rm SGN}_{e}(z)$, where $e$ is the vector of all ones. Together with the definition of ${\cal P}$, we could equivalently write \eqref{eq: subdiff_exclusive} as
\begin{align}
\partial \Delta^{\mathcal{G},w}(x)=\left\{u\in \mathbb{R}^n: u_{g_j}=2\|w_{g_j}\circ x_{g_j}\|_1 v_{g_j},\  v_{g_j} \in {\rm SGN}_{w_{g_j}}(x_{g_j})\right\}.\label{eq: subdiff_exclusive_2}
\end{align}

In addition to the subdifferential of the weighted exclusive lasso regularizer, we also summarize the characterization of $\partial p(x)$ for some other commonly used regularizers in Table \ref{tab: partial_p}. Note that the sum rule of subdifferential \cite[Theorem 23.8]{rockafellar1970convex} is used in the cases of the generalized lasso and the elastic net.

\begin{table}[H]\fontsize{8pt}{11pt}\selectfont
	\setlength{\abovecaptionskip}{0.2pt}
	\setlength{\belowcaptionskip}{0.2pt}
	\caption{Subdifferential of $p(\cdot)$.}\label{tab: partial_p}
	\renewcommand\arraystretch{1.2}
	\centering
	\begin{tabular}{|c|c|c|}
		\hline
		Name of regularizer & $p(\cdot)$ & $\partial p(\cdot)$  \\
		\hline
		lasso  & $p(x)=\|x\|_1$ & $\partial p(x)={\rm SGN}(x)$ \\
		\hline
		generalized lasso & $p(x)=\|x\|_1+\beta \|Bx\|_1$, $B\in \mathbb{R}^{s\times n}$ & $\partial p(x)=\{u+\beta B^T v\mid u\in {\rm SGN}(x),\ v\in {\rm SGN}(Bx)\}$\\
		\hline
		elastic net & $p(x)=\|x\|_1+\beta \|x\|_2^2$ & $\partial p(x)=\{u+2\beta x\mid u\in {\rm SGN}(x) \}$ \\
		\hline
		group lasso & $p(x)=\sum_{j=1}^l \beta_{j}\|x_{g_j}\|_2$ &  $\partial p(x)=\left\{u\in \mathbb{R}^n: u_{g_j}\in \left\{\begin{aligned}
		&\{\beta_{j}\frac{x_{g_j}}{\|x_{g_j}\|}\} && \mbox{if}\ \|x_{g_j}\|\neq 0 \\
		&\beta_{j} \mathbb{B}_2 && \mbox{otherwise}
		\end{aligned}\right.\right\}$\\
		\hline
	\end{tabular}
\end{table}

\paragraph{Algorithm for the subroutine.} Next we discuss how to solve the problem \eqref{eq: lambda0_problem} and the constrained problem \eqref{eq: constrained} in the AS strategy for the weighted exclusive lasso model, that is, $\Phi(x)=h(Ax)-\langle c,x\rangle$ and $p(x)=\Delta^{\mathcal{G},w}(x)$. The constrained minimization problem \eqref{eq: constrained} could equivalently be written as
\begin{align*}
x^{l+1}(\lambda_i) \in \underset{ x\in \mathbb{R}^n} {\arg\min} \  \Big\{h(Ax)-\langle c,x\rangle  + \lambda_i \sum_{j=1}^l \|w_{g_j\cap I^{l+1}(\lambda_i)}\circ x_{g_j\cap I^{l+1}(\lambda_i)}\|_1^2- \langle \delta_i^{l+1},x\rangle
\ \mid \ x_{\bar{I}^{l+1}(\lambda_i)}=0\Big\}.
\end{align*}
Define $\hat{w}=w_{I^{l+1}(\lambda_i)}$,
\begin{align*}
\hat{g}_j:=\{k\in \{1,\cdots,|I^{l+1}(\lambda_i)|\}: \mbox{the $k$-th element of } I^{l+1}(\lambda_i)\mbox{ belongs to } g_j\},\quad j = 1,\cdots,l.
\end{align*}
Note that $\widehat{\mathcal{G}} := \{\hat{g}_1,\cdots,\hat{g}_l\}$ is an index partition such that $\bigcup_{j=1}^l \hat{g}_j = \{1,2, \dots,|I^{l+1}(\lambda_i)|\}$ and $\hat{g}_j \bigcap \hat{g}_k = \emptyset$ for any $j\neq k$. We also define the proper closed convex function $\hat{p}:\mathbb{R}^{|I^{l+1}(\lambda_i)|}\rightarrow \mathbb{R}$ by
\begin{align*}
\hat{p}(z)=\sum_{j=1}^l \|\hat{w}_{\hat{g}_j}\circ z_{\hat{g}_j}\|_1^2,\quad \forall z\in\mathbb{R}^{|I^{l+1}(\lambda_i)|}.
\end{align*}
The presence of the error vector $\delta_i^{l+1}$ in \eqref{eq: constrained} indicates that the computation of $x^{l+1}(\lambda_i)$ can be interpreted as follows: compute
\begin{align}
z^*\approx \underset{ z\in \mathbb{R}^{|I^{l+1}(\lambda_i)|}} {\arg\min} \  \Big\{h(A_{I^{l+1}(\lambda_i)}z)-\langle c_{I^{l+1}(\lambda_i)},z\rangle  + \lambda_i \hat{p}(z)
\Big\},\label{eq: reduced}
\end{align}
approximately to the accuracy that $\|(\delta_i^{l+1})_{I^{l+1}(\lambda_i)}\|\leq \epsilon/\sqrt{2}$, where $A_I$ is the matrix consisting of the columns of $A$ indexed by $I$, and
\begin{align}
(\delta_i^{l+1})_{I^{l+1}(\lambda_i)}\in A_{I^{l+1}(\lambda_i)}^T \nabla h(A_{I^{l+1}(\lambda_i)} z^*)- c_{I^{l+1}(\lambda_i)}+\lambda_i \partial \hat{p}(z^*),\label{eq: reduced_accuracy}
\end{align}
then extend $z^*\in \mathbb{R}^{|I^{l+1}(\lambda_i)|}$ to $x^{l+1}(\lambda_i)\in \mathbb{R}^n$ as
\begin{align*}
\Big( x^{l+1}(\lambda_i)\Big)_{I^{l+1}(\lambda_i)}=z^*,\quad \Big( x^{l+1}(\lambda_i)\Big)_{\bar{I}^{l+1}(\lambda_i)}=0.
\end{align*}
The problem \eqref{eq: reduced} is in the same form as $(\mbox{P}_{\lambda_i})$ but with a smaller problem size. Therefore, all we need is an efficient and robust algorithm for solving the problems in the form of $(\mbox{P}_{\lambda})$. In this paper, we propose a dual Newton method based proximal point algorithm for solving the exclusive lasso models. Besides being theoretically efficient, in the numerical experiments, it is also demonstrated to be highly efficient and robust practically.

\section{A dual Newton method based proximal point algorithm}
\label{sec:pppa}

For the machine learning model \eqref{eq: general_ml_model} that we are interested in,
we write the problem \eqref{eq: lasso_model} as
\begin{align}
\label{eq: Convex-composite-program}
\min_{x\in \mathbb{R}^n} \ \Big\{ f(x):= h(Ax)  - \langle c,x\rangle + \lambda p(x)\Big\},
\end{align}
When we take $p(\cdot)=\Delta^{\mathcal{G},w}(\cdot)$, where $\Delta^{\mathcal{G},w}(\cdot)$ is the weighted exclusive lasso regularizer defined in \eqref{eq: exclusive-lasso-regularizer}. The problem \eqref{eq: Convex-composite-program} is the so-called weighted exclusive lasso model. According to the notation in Section \ref{sec: screening}, the optimal solution set to \eqref{eq: Convex-composite-program} is denoted as $\Omega_{\lambda}$, which is assumed to be nonempty and compact.

We aim to design a preconditioned proximal point algorithm (PPA) to solve the convex composite programming problem \eqref{eq: Convex-composite-program}. For the case when $p(\cdot)=\Delta^{\mathcal{G},w}(\cdot)$, we prove that ${\cal T}_f:=\partial f$ satisfies a certain error bound condition which guarantees that the proposed preconditioned PPA for solving the weighted exclusive lasso model has the asymptotic superlinear convergence property. As for the PPA subproblem, we design a dual Newton method, which is proved to have superlinear or even quadratic convergence.

\subsection{A preconditioned PPA for the weighted exclusive lasso model \eqref{eq: Convex-composite-program}}
For any starting point $x^0\in \mathbb{R}^n$, the preconditioned PPA generates a sequence $\{x^k\}\subseteq \mathbb{R}^n$ by the following approximate rule for solving \eqref{eq: Convex-composite-program}:
\begin{align}
x^{k+1}\approx {\cal P}_k(x^k):=\underset{x\in\mathbb{R}^n}{\arg\min}\ \Big\{f_k(x):=h(Ax)  - \langle c,x\rangle +\lambda  p(x)+\frac{1}{2\sigma_k}\|x-x^k\|_{{\cal M}}^2\Big\},\label{eq:pre_ppa}
\end{align}
where $\{\sigma_k\}$ is a sequence of nondecreasing positive real numbers $(\sigma_k\uparrow \sigma_{\infty}\leq \infty)$ and ${\cal M}=I_n+\tau A^T A$ for some given parameter $\tau> 0$. We should mention that the choice of this special ${\cal M}$ instead of the usual identity operator is important for us to obtain the dual of \eqref{eq:pre_ppa} as a smooth unconstrained problem. To ensure the convergence of the preconditioned PPA, we use the following stopping criteria:
\begin{align}
\|x^{k+1}-{\cal P}_k(x^k)\|_{{\cal M}}&\leq \epsilon_k,\quad \epsilon_k \geq 0,\quad \sum_{k=0}^{\infty}\epsilon_k <\infty,
\tag{A}\label{eq:stopA_pre_ppa}\\
\|x^{k+1}-{\cal P}_k(x^k)\|_{{\cal M}}&\leq \delta_k\|x^{k+1}-x^k\|_{{\cal M}},\quad 0\leq \delta_k < 1,\quad \sum_{k=0}^{\infty}\delta_k <\infty.\tag{B}\label{eq:stopB_pre_ppa}
\end{align}

The following theorem states the convergence results of the preconditioned PPA, which can be found in \cite{li2019asymptotically}. Note that the desired asymptotic superlinear convergence rate depends on the assumption that ${\cal T}_f= \partial f$ satisfies the error bound condition \eqref{error_bound}. We will prove later that for the case of $p(\cdot)=\Delta^{\mathcal{G},w}(\cdot)$, this assumption holds for commonly used loss functions.
\begin{theorem}\label{thm:convergence_ALM}
	(1) Let $\{x^k\}$ be the sequence generated by the preconditioned PPA \eqref{eq:pre_ppa} with the stopping criterion \eqref{eq:stopA_pre_ppa}. Then $\{x^k\}$ is bounded and $\{x^k\}$ converges to some $x^*\in \Omega_{\lambda}$.\\
	\noindent (2) Let $r:=\sum_{i=0}^{\infty}\epsilon_k+{\rm dist}_{{\cal M}}(x^0,\Omega_{\lambda})$. Assume that for this $r>0$, there exists a constant $\kappa>0$ such that ${\cal T}_f $ satisfies the following error bound condition
	\begin{align}
	{\rm dist}(x,\Omega_{\lambda})\leq \kappa {\rm dist}(0,{\cal T}_f(x)),\quad \mbox{$\forall x\in \mathbb{R}^n$ satisfying }{\rm dist}(x,\Omega_{\lambda})\leq r. \label{error_bound}
	\end{align}
	Suppose that $\{x^k\}$ is generated by the preconditioned PPA with the stopping criteria \eqref{eq:stopA_pre_ppa} and \eqref{eq:stopB_pre_ppa}. Then it holds for all $k\geq 0$ that
	\begin{align*}
	{\rm dist}_{{\cal M}}(x^{k+1},\Omega_{\lambda})\leq \mu_k {\rm dist}_{{\cal M}}(x^{k},\Omega_{\lambda}),
	\end{align*}
	where
	\[
	\mu_k=\frac{1}{1-\delta_k}\Big[\delta_k+ \frac{(1+\delta_k)\kappa\lambda_{\max}({\cal M}
		)}{\sqrt{\sigma_k^2+\kappa^2\lambda_{\max}^2(
			{\cal M})}}\Big]\rightarrow \mu_{\infty}=\frac{\kappa\lambda_{\max}({\cal M}
		)}{\sqrt{
			\sigma_{\infty}^2+\kappa^2\lambda_{\max}({\cal M}
			)^2}}<1,\quad k\rightarrow \infty,
	\]
	where $\lambda_{\max}({\cal M})=1+\tau\lambda_{\max}(A^T A)$.
\end{theorem}

\subsection{Error bound conditions for the weighted exclusive lasso model}
As one can see in Theorem \ref{thm:convergence_ALM}, the convergence rate of the preconditioned PPA relies on the error bound condition \eqref{error_bound} of ${\cal T}_f$. In this section, we first establish the error bound condition of ${\cal T}_f$ for the problem \eqref{eq: Convex-composite-program} with a piecewise linear-quadratic regularizer, based on the proximal residual function $R_{\lambda}(x)$ defined in \eqref{eq:residual}. Since $\Delta^{\mathcal{G},w}(x)=\sum_{j=1}^l \|w_{g_j} \circ x_{g_j}\|_1^2$ is piecewise linear-quadratic, the general result holds also for the exclusive lasso regularizer.

\begin{proposition}\label{prop:luotseng}
For the problem \eqref{eq: Convex-composite-program}, suppose that $h(\cdot)$ is strongly convex on any compact convex set in $\mathbb{R}^m$, $p(\cdot)$ is piecewise linear-quadratic. Then  for any $\xi\geq \inf_{x\in \mathbb{R}^n}f(x)$, there exist constants $\kappa,\varepsilon>0$ such that
\begin{align*}
{\rm dist}(x,\Omega_{\lambda})\leq \kappa \|R_{\lambda}(x)\|\mbox{ \quad  for all $x\in \mathbb{R}^n$ with $f(x)\leq \xi$, $\|R_{\lambda}(x)\|\leq \varepsilon$}.
\end{align*}
\end{proposition}
\begin{proof}
	Since $p$ is piecewise linear-quadratic, $p^*$ is also piecewise linear-quadratic by \cite[Theorem 11.14(b)]{rockafellar2009variational}. Thus $\partial p$ and $\partial p^*$ are both polyhedral due to \cite[Proposition 10.21]{rockafellar2009variational}. Define the solution map $\Gamma:\mathbb{R}^m\times \mathbb{R}^n\rightarrow \mathbb{R}^n$ as $\Gamma(y,g):=\{x\in \mathbb{R}^n\mid Ax=y,-g\in \partial p(x)\}$. Note that $\Gamma$ is a polyhedral multifunction, thus it is locally upper Lipschitz continuous at any $(y,g)\in \mathbb{R}^m\times \mathbb{R}^n$ by \cite{robinson1981some}. Therefore the desired conclusion holds by \cite[Corollary 1]{zhou2017unified}.
\end{proof}

In the next proposition, we prove that the error bound condition \eqref{error_bound} holds for the linear regression problem and the logistic regression problem with a piecewise linear-quadratic regularizer.
\begin{proposition} \label{prop:error-bound}
	Assume that $p(\cdot)$ is piecewise linear-quadratic. Then the error bound condition \eqref{error_bound} holds if $h(\cdot)$ is strongly convex on any compact convex set in $\mathbb{R}^m$. In particular, the latter property is satisfied by the following two special cases:
	\begin{itemize}
		\item[(1)] (linear regression) $h(y)=\sum_{i=1}^m (y_i-b_i)^2/2$, for some given vector $b\in \mathbb{R}^m$;
		\item[(2)] (logistic regression) $h(y) = \sum_{i=1}^m \log(1 + \exp(-b_i y_i))$, for some given vector $b\in \{-1,1\}^m$.
	\end{itemize}
\end{proposition}
\begin{proof}
	Given $r>0$, define
	\begin{align*}
	\widehat{\Omega}_r:=\{x\in \mathbb{R}^n\mid {\rm dist}(x,\Omega_{\lambda})\leq r\}.
	\end{align*}
	Due to the compactness of $\Omega_{\lambda}$, $\widehat{\Omega}_r$ is also compact and thus $\xi:=\max_{x\in \widehat{\Omega}_r}f(x)$ is finite. Note that the assumptions in Proposition \ref{prop:luotseng} are satisfied. For this $\xi$, there exist constants $\kappa,\varepsilon>0$ such that
	\begin{align}
	{\rm dist}(x,\Omega_{\lambda})\leq \kappa \|R_{\lambda}(x)\|\mbox{ \quad  for all $x\in\mathbb{R}^n$ with $f(x)\leq \xi$, $\|R_{\lambda}(x)\|\leq \varepsilon$}.\label{eq:luocondition}
	\end{align}
	For any $x\in \widehat{\Omega}_r$, we consider two cases: if $\|R_{\lambda}(x)\|\leq \varepsilon$, from \eqref{eq:luocondition}, we have ${\rm dist}(x,\Omega_{\lambda})\leq \kappa \|R_{\lambda}(x)\|$; if $\|R_{\lambda}(x)\|> \varepsilon$, ${\rm dist}(x,\Omega_{\lambda})\leq r= (r/\varepsilon)\varepsilon < (r/\varepsilon) \|R_{\lambda}(x)\|$. Therefore, it holds that
	\[
	{\rm dist}(x,\Omega_{\lambda})\leq \max\{\kappa,(r/\varepsilon)\} \|R_{\lambda}(x)\|,\quad \forall x\in \widehat{\Omega}_r.
	\]
	Next, consider an arbitrary $x\in \widehat{\Omega}_r.$ For any $y\in {\cal T}_f(x)$, we have that $x={\rm Prox}_{\lambda p}(x+y-A^T \nabla h(Ax)+c)$, and
	\begin{align*}
	\|R_{\lambda}(x)\|=\|{\rm Prox}_{\lambda p}(x+y-A^T \nabla h(Ax)+c)-{\rm Prox}_{\lambda p}(x-A^T \nabla h(A x)+c)\|\leq \|y\|.
	\end{align*}
	Therefore, ${\rm dist}(x,\Omega_{\lambda})\leq \max\{\kappa,(r/\varepsilon)\}
	\norm{y}$ for any $y\in {\cal T}_f(x)$. This implies that
	\begin{align*}
	{\rm dist}(x,\Omega_{\lambda})\leq \max\{\kappa,(r/\varepsilon)\}
	{\rm dist}(0,{\cal T}_f(x)).
	\end{align*}
Since $x\in \widehat{\Omega}_r$ is arbitrary, the above inequality implies that the error bound condition \eqref{error_bound} holds.
\end{proof}

From Proposition \ref{prop:error-bound}, the preconditioned PPA for solving the linear regression and logistic regression problems with the exclusive lasso regularizer is guaranteed to have fast linear convergence when the parameters $\{\sigma_k\}$ are large. Note that the key challenge in executing the preconditioned PPA is whether the nonsmooth problem \eqref{eq:pre_ppa} can be solved efficiently. Next we design a dual Newton method to solve it, which is expected to be superlinearly (or even quadratically) convergent.

\subsection{A dual Newton method for the preconditioned PPA subproblem \eqref{eq:pre_ppa}}
\label{sec:dna}
Note that in the preconditioned PPA subproblem \eqref{eq:pre_ppa}, $f_k(\cdot)$ is strongly convex and nonsmooth. Thus it admits a unique minimizer $\bar{x}^{k+1}$. The main point is how one can solve the minimization problem in a fast and robust way. Our choice is the dual Newton method, that is applying the semismooth Newton method for solving the dual problem of \eqref{eq:pre_ppa}.

Since ${\cal M}=I_n+\tau A^T A$, the dual of \eqref{eq:pre_ppa} can be shown to be expressible as the following smooth unconstrained problem:
\begin{align}
\max_{u\in \mathbb{R}^n}\ \Big\{\psi_k(u)&:= -\frac{\tau}{2\sigma_k}\|Ax^k+\frac{\sigma_k}{\tau}u\|^2 +\frac{\tau}{\sigma_k}{\rm E}_{\sigma_k h/\tau}(Ax^k+\frac{\sigma_k}{\tau}u) +\frac{\tau}{2\sigma_k}\|Ax^k\|^2\notag\\
&- \frac{1}{2\sigma_k}\|x^k+\sigma_k c-\sigma_k A^T u\|^2 +\frac{1}{\sigma_k}{\rm E}_{\sigma_k\lambda  p}(x^k+\sigma_k c-\sigma_kA^Tu)+\frac{1}{2\sigma_k}\|x^k\|^2\Big\}.\label{eq:ATA_D}
\end{align}
Since $h(\cdot)$ is convex and twice continuously differentiable, we can establish the following proposition.
\begin{proposition}\label{prop: jacobian_h}
	For any $\nu>0$, ${\rm Prox}_{\nu h}(z)$ is differentiable with
	\begin{align*}
	\nabla {\rm Prox}_{\nu h}(z)=(I_m+\nu \nabla^2 h({\rm Prox}_{\nu h}(z)))^{-1},\quad \forall z\in \mathbb{R}^m.
	\end{align*}
	Therefore, $0\prec \nabla {\rm Prox}_{\nu h}(z) \preceq I_m$ for any $z\in \mathbb{R}^m$, and thus the function
	\begin{align*}
	\theta(z):=\frac{1}{2}\|z\|^2-{\rm E}_{\nu h}(z)
	\end{align*}
	is strictly convex on $\mathbb{R}^m$.
\end{proposition}
\begin{proof}
	Define $F:\mathbb{R}^{2m}\rightarrow \mathbb{R}^m$ as
	\begin{align*}
	F(u,v) = v-u+\nu \nabla h(v),\quad \forall (u,v)\in \mathbb{R}^m\times\mathbb{R}^m.
	\end{align*}
	The optimality condition of $\min_w\{\frac{1}{2}\|w-z\|^2+\nu h(w)\}$ implies that for any $z\in \mathbb{R}^m$, there exists a unique $w$ such that $F(z,w)=0$,
	where this unique minimizer $w$ is denoted as ${\rm Prox}_{\nu h}(z)$. Let the Jacobian of $F$ with respect to $u$ and $v$ be denoted as $J_{F,u}$ and $J_{F,v}$, respectively. We have that
	\begin{align*}
	J_{F,v}(z,w)=I_m +\nu \nabla^2 h(w)
	\end{align*}
	is invertible. According to the implicit function theorem, we know that there exists an open set $U\subset \mathbb{R}^m$ containing $z$ such that there exists a unique continuously differentiable function $g: U\rightarrow \mathbb{R}^m$ such that $g(z)=w$ and
	\begin{align*}
	&F(u,g(u))=g(u)-u+\nu \nabla h(g(u))=0,\quad \forall u\in U,\\
	&\nabla g(u)=-[J_{F,v}(u,g(u))]^{-1}J_{F,u}(u,g(u))=(I_m +\nu \nabla^2 h(g(u)))^{-1},\quad \forall u\in U.
	\end{align*}
	Combining the uniqueness of the function $g(\cdot)$ and the definition of ${\rm Prox}_{\nu h}(\cdot)$, we have that ${\rm Prox}_{\nu h}(u)=g(u)$ for all $u\in U$ and
	\begin{align*}
	\nabla {\rm Prox}_{\nu h}(z )= (I_m +\nu \nabla^2 h({\rm Prox}_{\nu h}(z)))^{-1}.
	\end{align*}
	The remaining part of the conclusion follows naturally since $\nabla \theta (z)= {\rm Prox}_{\nu h}(z )$.
\end{proof}

The above proposition shows that $\psi_k(\cdot)$ is strictly concave. Thus it admits a unique maximizer, which we denote as $\bar{u}^{k+1}$. As long as we obtain $\bar{u}^{k+1}$, the update of $x$ in the preconditioned PPA iteration \eqref{eq:pre_ppa} could be given as
\begin{align*}
\bar{x}^{k+1}={\rm Prox}_{\sigma_k \lambda p}(x^k + \sigma_k c - \sigma_k A^T\bar{u}^{k+1}).
\end{align*}
Since $\psi_k$ is continuously differentiable, $\bar{u}^{k+1}$ could be computed by solving
\begin{align}
\nabla \psi_k(u)=-{\rm Prox}_{\sigma_k h/\tau}(Ax^k+\frac{\sigma_k}{\tau}u)+A {\rm Prox}_{\sigma_k \lambda p}(x^k+\sigma_k c-\sigma_kA^Tu)=0.\label{eq: newton-system}
\end{align}
Note that $\nabla \psi_k(\cdot)$ is Lipschitz continuous, but nondifferentiable. We propose a semismooth Newton (SSN) method to solve the nonsmooth equation \eqref{eq: newton-system}, which is proved to have at least superlinear convergence. The concept of semismoothness could be found in \cite{kummer1988newton,mifflin1977semismooth,qi1993nonsmooth,sun2002semismooth}.

Define the multifunction $\hat{\partial}^2 \psi_k(\cdot):\mathbb{R}^m\rightrightarrows \mathbb{R}^{m\times m}$ as follows: for any $u\in \mathbb{R}^m$,
\begin{align}
\hat{\partial}^2 \psi_k(u)  :=  -\frac{\sigma_k}{\tau}\nabla {\rm Prox}_{\sigma_k h/\tau}(Ax^k+\frac{\sigma_k}{\tau}u)  - \sigma_kA\partial_{\rm HS}{\rm Prox}_{\sigma_k \lambda p}(x^k + \sigma_k c - \sigma_kA^Tu)A^T,
\label{eq: jacobian}
\end{align}
where $\partial_{\rm HS} {\rm Prox}_{\sigma_k \lambda  p}(\cdot)$ is defined in Proposition \ref{prop: proxmapping_and_jacobian}. The following proposition states that $\hat{\partial}^2 \psi_k(\cdot)$ could be treated as the generalized Jacobian of the Lipschitz continuous function $\nabla\psi_k(\cdot)$.
\begin{proposition}\label{prop: jacobian_psi}
	The multifunction $\hat{\partial}^2 \psi_k(\cdot)$ defined in \eqref{eq: jacobian} satisfies the following properties:
	\begin{itemize}
		\item[(1)] $\hat{\partial}^2 \psi_k(\cdot)$ is a nonempty, compact valued, upper-semicontinuous multifunction;
		\item[(2)] for any $u\in \mathbb{R}^m$, all the elements in $\hat{\partial}^2 \psi_k(u) $ are symmetric and negative definite;
		\item[(3)] $\nabla \psi_k(\cdot)$ is strongly semismooth with respect to $\hat{\partial}^2 \psi_k(\cdot)$.
	\end{itemize}
\end{proposition}
\begin{proof}
	The conclusions follows from Proposition \ref{prop: jacobian_h} and Proposition \ref{prop: proxmapping_and_jacobian}. We omit the details here.
\end{proof}

Now we present the SSN method for solving \eqref{eq:ATA_D} in Algorithm \ref{alg:ssn}.
\begin{algorithm}[H]\small
	\caption{Semismooth Newton method for \eqref{eq:ATA_D}}
	\label{alg:ssn}
	\begin{algorithmic}[1]
		\STATE \textbf{Input}: $\mu \in (0, 1/2)$, $\bar{\tau} \in (0, 1]$, and $\bar{\gamma}, \delta \in (0, 1)$.
		\STATE \textbf{Output}: an approximate optimal solution $u^{k+1}$ to \eqref{eq:ATA_D}.
		\STATE \textbf{Initialization}: choose $u^{k,0} \in  \mathbb{R}^m$, $j=0$.
		\REPEAT
		\STATE {\bfseries Step 1}. Select an element ${\cal H}_j \in \hat{\partial}^{2} \psi_k(u^{k,j})$. Apply the direct method or the conjugate gradient (CG) method to find an approximate solution $d^j \in \mathbb{R}^m$ to
		\begin{align}\label{eq: cg-system}
		{\cal H}_j(d^j) \approx - \nabla \psi_k(u^{k,j}),
		\end{align}
		such that $\|{\cal H}_j(d^j) + \nabla\psi_k(u^{k,j})\| \leq \min(\bar{\gamma}, \|\nabla\psi(u^{k,j})\|^{1+\bar{\tau}})$.
		\\[3pt]
		\STATE {\bfseries Step 2}. Set $\alpha_j = \delta^{m_j}$, where $m_j$ is the smallest nonnegative integer $m$ for which
		$$\psi_k(u^{k,j} + \delta^m d^j) \geq \psi_k(u^{k,j}) + \mu\delta^m \langle \nabla\psi_k(u^{k,j}), d^j \rangle .$$
		\\[3pt]
		\STATE{\bfseries Step 3}. Set $u^{k,j+1} = u^{k,j} + \alpha_j d^j$, $u^{k+1}=u^{k,j+1}$, $j\leftarrow j+1$.
		\UNTIL{Stopping criterion based on $u^{k+1}$ is satisfied.}
	\end{algorithmic}
\end{algorithm}

The following theorem gives the convergence result of the SSN method, which could be proved by using Proposition \ref{prop: jacobian_psi} and the results in\cite[Proposition 3.3 and Theorem 3.4]{zhao2010newton}, \cite[Theorem 3]{li2018efficiently}.
\begin{theorem} \label{thm:convergence_SSN}
	Let $\{u^{k,j}\}$ be the sequence generated by Algorithm \ref{alg:ssn}. Then $\{u^{k,j}\}$ converges to the unique optimal solution $\bar{u}^{k+1}$ of the problem \eqref{eq:ATA_D}, and for $j$ sufficiently large,
	\begin{align*}
	\|u^{k,j+1} - \bar{u}^{k+1}\| = O(\|u^{k,j} - \bar{u}^{k+1}\|^{1+\bar{\tau}}),
	\end{align*}
	where $\bar{\tau} \in (0, 1]$ is given in the algorithm.
\end{theorem}

We should emphasize that the efficiency of computing the Newton direction in \eqref{eq: cg-system} depends critically on exploiting the sparsity structure of the generalized Jacobian. The numerical implementation details of the SSN method for the case of $p(\cdot)=\Delta^{\mathcal{G},w}(\cdot)$ could be found in the appendix.

\subsection{The PPDNA framework for solving weighted exclusive lasso models}
We give the full description of the PPDNA in Algorithm \ref{alg:ppdna}.
\begin{algorithm}[H]\small
	\caption{Dual Newton method based proximal point algorithm for \eqref{eq: Convex-composite-program}}
	\label{alg:ppdna}
	\begin{algorithmic}[1]
		\STATE \textbf{Input}: $0<\sigma_0\leq \sigma_{\infty}\leq \infty$.
		\STATE \textbf{Output}: an approximate optimal solution $x$ to \eqref{eq: Convex-composite-program}.
		\STATE \textbf{Initialization}: choose $x^0\in \mathbb{R}^n$, $k=0$.
		\REPEAT
		\STATE {\bfseries Step 1}. Apply Algorithm \ref{alg:ssn} to maximize $\psi_{k}(u)$, where $\psi_k(\cdot)$ is defined as in \eqref{eq:ATA_D}. Based on the updated $u^{k+1}$ in {\bf Step 3} of Algorithm \ref{alg:ssn}, we also compute
		\begin{align*}
		x^{k+1} = {\rm Prox}_{\sigma_k\lambda  p}(x^k+\sigma_k c-\sigma_k A^Tu^{k+1}).
		\end{align*}
		Algorithm \ref{alg:ssn} is terminated if the stopping criteria \eqref{eq:stopA_pre_ppa} and \eqref{eq:stopB_pre_ppa} are satisfied.
		\\[3pt]
		\STATE {\bfseries Step 2}. Update $\sigma_{k+1} \uparrow \sigma_{\infty} \leq \infty$, $k\leftarrow k+1$.
		\UNTIL{Stopping criterion is satisfied.}
	\end{algorithmic}
\end{algorithm}

Note that in practice, it is difficult to directly check the stopping criteria \eqref{eq:stopA_pre_ppa} and \eqref{eq:stopB_pre_ppa}. We prove that they can be achieved by other implementable criteria based on $u^{k+1}$ and $x^{k+1}$ as follows.
\begin{proposition}\label{prop:stop_of_ATA}
	In Algorithm \ref{alg:ppdna}, the stopping criteria \eqref{eq:stopA_pre_ppa} and \eqref{eq:stopB_pre_ppa} can be achieved by the following two implementable ones:
	\begin{align}
	f_k(x^{k+1})-\psi_k(u^{k+1})&\leq \frac{\epsilon_k^2}{2\sigma_k},\quad \epsilon_k \geq 0,\quad \sum_{k=0}^{\infty}\epsilon_k <\infty,\tag{A'}\label{eq:stopA_ATA}\\
	f_k(x^{k+1})-\psi_k(u^{k+1})&\leq \frac{\delta_k^2}{2\sigma_k} \|x^{k+1}-x^k\|_{{\cal M}}^2,\quad 0\leq \delta_k < 1,\quad \sum_{k=0}^{\infty}\delta_k <\infty,\tag{B'}\label{eq:stopB_ATA}
	\end{align}
	where $f_k(\cdot)$ is defined in \eqref{eq:pre_ppa} and $\psi_k(\cdot)$ is defined in \eqref{eq:ATA_D}.
\end{proposition}
\begin{proof}
	We know from \cite[Exercise 8.8]{rockafellar2009variational} that $\partial f_k(x)=\partial f(x)+(1/\sigma_k) {\cal M}(x-x^k)$. Since ${\cal P}_k(x^k)=\arg\min f_k(x)$, we have that $0\in \partial f_k({\cal P}_k(x^k))$, which means there exists $v\in \partial f({\cal P}_k(x^k))$ such that
	\begin{align*}
	0=v+\frac{1}{\sigma_k} {\cal M}({\cal P}_k(x^k)-x^k).
	\end{align*}
	Since $f_k({\cal P}_k(x^k))=\inf f_k$, it holds that
	\begin{align*}
	&f_k(x^{k+1})-\inf f_k=f(x^{k+1})-f({\cal P}_k(x^k))+\frac{1}{2\sigma_k}\|x^{k+1}-x^k\|_{{\cal M}}^2-\frac{1}{2\sigma_k}\|{\cal P}_k(x^k)-x^k\|_{{\cal M}}^2\\
	&\geq \langle v,x^{k+1}-{\cal P}_k(x^k)\rangle+\frac{1}{2\sigma_k}\langle x^{k+1}+{\cal P}_k(x^k)-2x_k,x^{k+1}-{\cal P}_k(x^k)\rangle_{{\cal M}}=\frac{1}{2\sigma_k}\|x^{k+1}-{\cal P}_k(x^k)\|_{{\cal M}}^2.
	\end{align*}
	By the strongly duality, we know that $\inf f_k=\sup \psi_k$, thus
	\begin{align*}
	\frac{1}{2\sigma_k}\|x^{k+1}-{\cal P}_k(x^k)\|_{{\cal M}}^2\leq f_k(x^{k+1})-\inf f_k=f_k(x^{k+1})-\sup \psi_k\leq f_k(x^{k+1})-\psi_k(u^{k+1}).
	\end{align*}
	Therefore, the stopping criteria \eqref{eq:stopA_pre_ppa} and \eqref{eq:stopB_pre_ppa} can be achieved by \eqref{eq:stopA_ATA} and \eqref{eq:stopB_ATA}, respectively.
\end{proof}

\section{Closed-form solution to the proximal mapping of the weighted exclusive lasso regularizer $\Delta^{\mathcal{G},w}(\cdot)$ and its generalized Jacobian}
\label{sec:proxJacobian}
When solving the weighted exclusive lasso model with the PPDNA, it is clear that we need the proximal mapping of the weighted exclusive lasso regularizer and its generalized Jacobian. In this section, we give a systematically study of the weighted exclusive lasso regularizer. Specifically, we derive the closed-form solution to the proximal mapping ${\rm Prox}_{\nu p}(\cdot)$ for any $\nu>0$ with $p(\cdot)=\Delta^{\mathcal{G},w}(\cdot)$, and characterize the corresponding generalized Jacobian. Note that by the definition of $\Delta^{\mathcal{G},w}(\cdot)$, it is fundamental for us to study ${\rm Prox}_{\rho\|w\circ\cdot\|_1^2}(\cdot)$ first, where $w\in \mathbb{R}^n_{++}$ is a given weight vector and $\rho>0$ is a given scalar. 

\subsection{Closed-form solution to ${\rm Prox}_{\rho\|w\circ\cdot\|_1^2}(\cdot)$}
\label{subsec:prox}
For any $a\in \mathbb{R}^n$, we define
\begin{align}
x(a)&:=\underset{x \in \mathbb{R}^n_{+}}{\arg\min}\ \Big\{\frac{1}{2}\|x - a\|^2 + \rho\|w\circ x\|_1^2
\Big\}=\underset{x \in \mathbb{R}^n_{+}}{\arg\min} \ \Big\{\frac{1}{2}\|x - a\|^2 +  \rho x^T(ww^T)x \Big\}.\label{eq:pos_problem}
\end{align}
Our derivation of ${\rm Prox}_{\rho\|w\circ\cdot\|_1^2}(\cdot)$ is based on the analysis of the KKT optimality conditions. 
We first consider the case when the input vector $a\geq 0$, then one can see that ${\rm Prox}_{\rho\|w\circ\cdot\|_1^2}(a)$ must also be a nonnegative vector, which means that ${\rm Prox}_{\rho\|w\circ\cdot\|_1^2}(a)=x(a)$ for any $a\geq 0$.

Note that since the objective function in \eqref{eq:pos_problem} is strongly convex, the minimization problem has a unique solution, which can be computed as in the following proposition.
\begin{proposition}\label{prop:pos_prox}
	Given $\rho > 0$ and $a \in \mathbb{R}_{+}^n \backslash\{0\}$. Let $a^{w}\in \mathbb{R}^n$ be defined as $a^{w}_i :=a_i/w_i$, for $i=1,\cdots,n$. There exists a permutation matrix $\Pi$ such that $\Pi a^{w}$ is sorted in a non-increasing order. Denote $\tilde{a}=\Pi a $,  $\tilde{w}=\Pi w $, and
	\begin{align*}
	s_{i} = \sum_{j=1}^{i}
	\tilde{w}_j \tilde{a}_j, \quad
	L_{i} = \sum_{j=1}^{i}\tilde{w}_j^2,\quad
	\alpha_i = \frac{  s_{i}}{1 + 2\rho L_{i}}, \quad i = 1, 2, \dots, n.
	\end{align*}
	Let $\bar{\alpha}=\max_{1\leq i \leq n}\alpha_i$. Then, $x(a)$ in \eqref{eq:pos_problem} can be computed analytically as $x(a) = (a -2\rho \bar{\alpha} w)^+$.
\end{proposition}
\begin{proof}
	The KKT conditions for \eqref{eq:pos_problem} are given by
	\begin{equation}\label{eq: pos_KKT}
	x - a + 2\rho w w^Tx + \mu = 0,\;
	\mu \circ x =0,\; \mu \leq 0, \;x \geq 0,
	\end{equation}
	where $\mu \in \mathbb{R}^n$ is the corresponding dual multiplier. If $(x^*,\mu^{*})$ satisfies the KKT conditions \eqref{eq: pos_KKT},  by denoting $\beta=w^T x^*$, we can see that
	\begin{equation*}
	x^* +  \mu^* = a - 2\rho\beta w ,\;
	\mu^* \circ x^* =0,\; \mu^* \leq 0, \;x^* \geq 0.
	\end{equation*}
	Therefore, $(x^*,\mu^{*})$ have the representations:
	\[
	x^*=(a - 2\rho\beta w)^+,\quad \mu^*=(a - 2\rho\beta w)^-.
	\]
	Then our aim is to find the value of $\beta$. By the definition of $\beta$, we can see that
	\begin{align*}
	\beta = \sum_{i=1}^n w_i x^*_i=\sum_{i=1}^n w_i (a_i - 2\rho\beta w_i)^+=\sum_{i=1}^n w^2_i ((a^{w})_i - 2\rho\beta )^+=\sum_{i=1}^n \tilde{w}^2_i ((\Pi a^{w})_i - 2\rho\beta )^+.
	\end{align*}
	Note that there must exist some index $j$ such that $(\Pi  a^{w})_j>2\rho \beta$, otherwise, we have $\beta=0$ and $\Pi a^{w}\leq 0$ (equivalent to $a\leq 0$), which contradicts the assumption that $0\not=a \geq 0$. Since $\Pi a^{w}$ is sorted in a non-increasing order, there exists an index $k$ such that $\tilde{a}_{1}/\tilde{w}_{1} \geq \dots \geq \tilde{a}_{k}/\tilde{w}_{k} \geq 2\rho\beta > \tilde{a}_{k+1}/\tilde{w}_{k+1} \geq \dots \geq \tilde{a}_{n}/\tilde{w}_{n}$. Therefore,
	\[
	\beta = \sum_{i=1}^{k} \tilde{w}^2_i ((\Pi a^{w})_i - 2\rho\beta )=\sum_{i=1}^{k}  \tilde{w}_i\tilde{a}_i- 2\rho\beta \sum_{i=1}^{k}  \tilde{w}^2_i=s_{k} -2\rho \beta L_{k} ,
	\]
	which means that
	\[
	\beta = \frac{s_{k} }{1+2\rho L_{k} }=\alpha_k.
	\]
	Next we show that $\beta=\bar{\alpha}$, which means $\alpha_k\geq \alpha_i$ for all $i$. For $i<k$,
	\begin{align*}
	&\ \alpha_k - \alpha_{i} =  \frac{(1 + 2\rho L_{i}) s_{k} - (1+2\rho L_{k}) s_{i}}{(1+2\rho L_{k})(1+2\rho L_{i})}=  \frac{(1 + 2\rho L_{k})(s_k-s_i) -2\rho s_k\sum_{j=i+1}^k\tilde{w}_j^2}{(1+2\rho L_{k})(1+2\rho L_{i})}\\
	&=  \frac{(1 + 2\rho L_{k})\sum_{j=i+1}^k\tilde{w}_j\tilde{a}_j -2\rho (1 + 2\rho L_{k})\beta\sum_{j=i+1}^k\tilde{w}_j^2}{(1+2\rho L_{k})(1+2\rho L_{i})}=  \frac{\sum_{j=i+1}^k \tilde{w}_j^2 (\tilde{a}_j/\tilde{w}_j -2\rho \beta)}{1+2\rho L_{i}}\geq 0.
	\end{align*}
	We can prove that $\alpha_k\geq \alpha_i$ for all $i>k$ in a similar way. Therefore, we have that $\beta=\alpha_k=\max_{1\leq i \leq n}\alpha_i=\bar{\alpha}$. Finally, since the solution to \eqref{eq:pos_problem} is unique, we have
	\begin{equation*}
	x(a)=x^*=(a - 2\rho\beta w)^+=(a - 2\rho\bar{\alpha} w)^+.\qedhere
	\end{equation*}
\end{proof}

With the result above, we now give the closed-form solution to ${\rm Prox}_{\rho\|w\circ\cdot\|_1^2}(a)$ for any $a \in \mathbb{R}^n$.
\begin{proposition}\label{prop:proxmapping}
	For given $\rho > 0$ and $a \in \mathbb{R}^n$, we have
	\[
	{\rm Prox}_{\rho\|w\circ\cdot\|_1^2}(a)={\rm sign}(a)\circ {\rm Prox}_{\rho\|w\circ\cdot\|_1^2}(|a|)={\rm sign}(a)\circ x(|a|),
	\]
	where $x(\cdot)$ is defined in \eqref{eq:pos_problem} and can be computed explicitly via Proposition \ref{prop:pos_prox} in $O(n\log n)$ operations.
\end{proposition}

\begin{remark}
	The proximal mapping of $\rho\|w\circ\cdot\|_1^2$ is mentioned in \cite[Proposition 4]{kowalski2009sparse}, but the derivation contains some errors. More precisely, in section 4.1 of \cite{kowalski2009sparse}, after a change of variables, the author tries to find the optimal solution of a constrained optimization problem by directly setting the gradient to zero (equations (23) and (24) in \cite{kowalski2009sparse}), which is not mathematically rigorous. One can use a simple example to demonstrate the gap. Consider the problem
\begin{align*}
\min_{x_{1,1},x_{1,2}\in\mathbb{R}}\Big\{\frac{1}{2}(x_{1,1}-1)^2+\frac{1}{2}(x_{1,2}-0.5)^2+(|x_{1,1}|+|x_{1,2}|)^2\Big\}.
\end{align*}
The true solution is $x^*=[1/3;0]$. But equation (23) in \cite{kowalski2009sparse} is equivalent to
\begin{align*}
|x_{1,1}| = 1-2(|x_{1,1}|+|x_{1,2}|),\quad
|x_{1,2}| = 0.5-2(|x_{1,1}|+|x_{1,2}|).
\end{align*}
Thus $|x_{1,1}|=2/5$, $|x_{1,2}|=-1/10$. But the latter contradicts the fact that $|x_{1,2}|\geq 0$.
\end{remark}

\subsection{The generalized Jacobian of ${\rm Prox}_{\rho\|w\circ\cdot\|_1^2}(\cdot)$ }
Note that in the SSN method (Algorithm \ref{alg:ssn}), it is critical for us to derive an explicit element in the generalized Jacobian of ${\rm Prox}_{\rho\|w\circ\cdot\|_1^2}(\cdot)$. According to Proposition \ref{prop:proxmapping}, we know that in order to obtain the generalized Jacobian of ${\rm Prox}_{\rho\|w\circ\cdot\|_1^2}(\cdot)$, we need to study the generalized Jacobian of $x(\cdot)$ first. Here, we derive the HS-Jacobian of $x(\cdot)$ based on the quadratic programming (QP) reformulation of $x(\cdot)$. For any $a\in \mathbb{R}^n$, if we denote $ Q := I_n +  2\rho ww^T \in \mathbb{R}^{n\times n}$, then \eqref{eq:pos_problem} can be equivalently written as
\begin{align*}
x(a)=\underset{x \in \mathbb{R}^n}{\arg\min} \ \Big\{ \frac{1}{2} \langle x, Qx \rangle - \langle x, a \rangle \mid x\geq 0 \Big\}.
\end{align*}
Based on the above strongly convex QP, we can derive the HS-Jacobian of $x(a)$ by applying the general results established in \cite{han1997newton,li2018efficiently}. As one can see from \eqref{eq: pos_KKT} and the uniqueness of $x(a)$, the dual multiplier $\mu$ is also unique, which we denote as $\mu(a)$. Then \eqref{eq: pos_KKT} can be equivalently written as
\begin{align*}
Qx(a) - a + \mu(a) = 0,\;
\mu(a)^Tx(a) = 0,\;\mu(a) \leq 0, \; x(a)\geq 0.
\end{align*}
Denote the active set of $x(a)$ as
\begin{align}
I(a): = {\{i \in \{ 1,\ldots,n \} \mid x(a)_i = 0\}}.
\label{eq:Ia}
\end{align}
Now, we define a collection of index sets:
\[
{\cal K}(a): =  \{\;  K\subseteq \{1,\ldots,n\} \mid  {\rm supp}(\mu(a)) \subseteq K\subseteq I(a)\},
\]
where ${\rm supp}(\mu(a))$ denotes  the set of indices $i$ such that $\mu(a)_i \neq 0$. Note that the set ${\cal K}(a)$ is non-empty \cite{han1997newton}. Since the B-subdifferential $\partial_B x(a)$ is difficult to compute, we define the multifunction $\partial_{\rm HS}x(\cdot)$: $\mathbb{R}^{n} 	\rightrightarrows \mathbb{R}^{n\times n}$ as
\begin{equation}\label{eq:HS_QP}
\partial_{\rm HS}x(a) := \left\{
P\in\mathbb{R}^{n\times n}\mid P = Q^{-1} - Q^{-1}
I_K^T\left( I_K Q^{-1}I_K^T\right)^{-1}I_K Q^{-1},\; K\in{\cal K}(a)
\right\},
\end{equation}
as a computational replacement for $\partial_Bx(a)$, where $I_K$ is the matrix consisting of the rows of $I_n$, indexed by $K$. The set $\partial_{\rm HS}x(a)$ is known as the HS-Jacobian  of $x(\cdot)$ at $a$.

Define the multifunction $\partial_{\rm HS} {\rm Prox}_{\rho\|w\circ\cdot\|_1^2}:\mathbb{R}^n\rightrightarrows \mathbb{R}^{n\times n}$ by
\begin{align}
\partial_{\rm HS} {\rm Prox}_{\rho\|w\circ\cdot\|_1^2}(a)=\left\{
{\rm Diag}(\theta ) P {\rm Diag}(\theta ) \mid \theta\in {\rm SGN}(a),\ P \in \partial_{\rm HS}x(|a|)
\right\}, \; \forall\; a\in \mathbb{R}^n,
\label{eq: Jacobian_prox_12}
\end{align}
where $\partial_{\rm HS}x(\cdot)$ is defined in \eqref{eq:HS_QP}, and ${\rm SGN}(\cdot)$ is defined in \eqref{eq: define_SGN}. The next proposition states the reason why we can treat $\partial_{\rm HS} {\rm Prox}_{\rho\|w\circ\cdot\|_1^2}(a)$ as the generalized Jacobian of ${\rm Prox}_{\rho\|w\circ\cdot\|_1^2}(\cdot)$ at $a$.
\begin{proposition}\label{prop: property_HS_prox}
	$\partial_{\rm HS} {\rm Prox}_{\rho\|w\circ\cdot\|_1^2}(\cdot)$ is a nonempty, compact valued and upper-semicontinuous multifunction. For any $a\in \mathbb{R}^n$, the elements in $\partial_{\rm HS} {\rm Prox}_{\rho\|w\circ\cdot\|_1^2}(a)$ are all symmetric and positive semidefinite. Moreover, ${\rm Prox}_{\rho\|w\circ\cdot\|_1^2}(\cdot)$ is strongly semismooth with respect to $\partial_{\rm HS} {\rm Prox}_{\rho\|w\circ\cdot\|_1^2}(\cdot)$.
\end{proposition}
\begin{proof}
	By the definition of $\partial_{\rm HS} {\rm Prox}_{\rho\|w\circ\cdot\|_1^2}(\cdot)$, we can see that it is a nonempty and compact valued multifunction. For $a\in \mathbb{R}^n$, according to \cite[Proposition 2]{li2018efficiently}, we know that there exists a neighborhood $U$ of $a$ such that for any $a' \in U$, $\partial_{\rm HS}x(|a'|)\subseteq \partial_{\rm HS}x(|a|)$ and
	\begin{align}
	x(|a'|) - x(|a|) - P(|a'| - |a|)=0,\quad \forall P \in \partial_{\rm HS}x(|a'|).\label{eq: strongly_xa}
	\end{align}
	By the definition of ${\rm SGN}(\cdot)$, if we take the neighbourhood $U$ of $a$ to be sufficiently small, then we have ${\rm SGN}(a')\subseteq {\rm SGN}(a)$ for any $a'\in U$. Therefore, it holds that $\partial_{\rm HS} {\rm Prox}_{\rho\|w\circ\cdot\|_1^2}(a')\subseteq \partial_{\rm HS} {\rm Prox}_{\rho\|w\circ\cdot\|_1^2}(a)$ for all $a'\in U$, which implies that $\partial_{\rm HS} {\rm Prox}_{\rho\|w\circ\cdot\|_1^2}$ is upper-semicontinuous.
	
	Since ${\rm Prox}_{\rho\|w\circ\cdot\|_1^2}(\cdot)$ is piecewise linear and Lipschitz continuous, it is directionally differentiable according to \cite{facchinei2007finite}. Note that for all $a'\in U$, since ${\rm SGN}(a')\subseteq {\rm SGN}(a)$, we have ${\rm Diag}(\theta) (a'-a)=|a'|-|a|$ with any $\theta \in {\rm SGN}(a')$. Therefore, from \eqref{eq: strongly_xa}, it holds that for any $a'\in U$,
	\begin{align*}
	\theta \circ x(|a'|) - \theta \circ x(|a|) - {\rm Diag}(\theta )P{\rm Diag}(\theta )(a'-a)=0,\quad \forall
	\theta \in {\rm SGN}(a'),\ \forall P \in \partial_{\rm HS}x(|a'|).
	\end{align*}
	By the formula of $x(\cdot)$ in Proposition \ref{prop:pos_prox}, for any $i$, if $(x(|a'|))_i\neq 0$, then we must have $a'_i\neq 0$, which further implies $\theta_i={\rm sign}(a'_i)$ for each $\theta\in {\rm SGN}(a')$. Therefore, we know that for all $a'\in U$,
	\begin{align*}
	\theta \circ x(|a'|) = {\rm sign}(a')\circ x(|a'|),\quad \theta \circ x(|a|) = {\rm sign}(a)\circ x(|a|),\quad \forall \theta \in {\rm SGN}(a')\subseteq {\rm SGN}(a).
	\end{align*}
    That is to say, when $a'\in U$,
	\begin{align*}
	{\rm Prox}_{\rho\|w\circ\cdot\|_1^2}(a')-{\rm Prox}_{\rho\|w\circ\cdot\|_1^2}(a)-M(a'-a)=0,\quad \forall M\in \partial_{\rm HS} {\rm Prox}_{\rho\|w\circ\cdot\|_1^2}(a').
	\end{align*}
	Thus ${\rm Prox}_{\rho\|w\circ\cdot\|_1^2}(\cdot)$ is strongly semismooth with respect to $\partial_{\rm HS} {\rm Prox}_{\rho\|w\circ\cdot\|_1^2}(\cdot)$.

    The symmetry of the elements in $\partial_{\rm HS} {\rm Prox}_{\rho\|w\circ\cdot\|_1^2}(a)$ for any $a\in \mathbb{R}^n$ follows naturally by the definition in \eqref{eq: Jacobian_prox_12}. In order to prove that the elements in $\partial_{\rm HS} {\rm Prox}_{\rho\|w\circ\cdot\|_1^2}(a)$ are all positive semidefinite for any $a\in \mathbb{R}^n$, it suffices to prove that the elements in $\partial_{\rm HS}x(a)$ are all positive semidefinite for any $a\in \mathbb{R}^n$. Given $a\in \mathbb{R}^n$, for any $K\in {\cal K}(a)$, denote $\xi\in \mathbb{R}^{n}$ with $\xi_i = 0$ if $i\in K$, and $\xi_i = 1$ otherwise. Let $\Xi = I_n-{\rm Diag}(\xi)$. After some algebraic multiplications, we can see that
    \begin{align*}
    I_K^T\left( I_K Q^{-1}I_K^T\right)^{-1}I_K=(\Xi Q^{-1}\Xi)^{\dagger}=\Xi (\Xi Q^{-1}\Xi)^{\dagger}\Xi,
    \end{align*}
    where the last equality follows from the fact that $\Xi$ is a $0$-$1$ diagonal matrix. Then by \cite[Proposition 3]{li2018efficiently}, we have
    \begin{align*}
    P&:=Q^{-1} - Q^{-1}I_K^T\left( I_K Q^{-1}I_K^T\right)^{-1}I_K Q^{-1}\\
    &=Q^{-1} - Q^{-1}\Xi (\Xi Q^{-1}\Xi)^{\dagger}\Xi Q^{-1}= ({\rm Diag}(\xi) Q {\rm Diag}(\xi))^{\dagger}\succeq 0,
    \end{align*}
    which completes the proof.
\end{proof}

For practical implementation of the SSN method, we need a specific element in $\partial_{\rm HS} {\rm Prox}_{\rho\|w\circ\cdot\|_1^2}(a)$ at any $a \in \mathbb{R}^n$. In the following proposition, we provide a highly efficient way to compute one specific element in $\partial_{\rm HS} {\rm Prox}_{\rho\|w\circ\cdot\|_1^2}(a)$, which is a $0$-$1$ diagonal matrix plus a rank-one correction.
\begin{proposition}\label{compute_M}
	Given $a \in \mathbb{R}^n$, if $P_0 :=  Q^{-1} - Q^{-1}I_{I(|a|)}^T
	(I_{I(|a|)} Q^{-1}
	I_{I(|a|)}^T)^{-1}
	I_{I(|a|)}Q^{-1}$,
	where $I(\cdot)$ is defined in \eqref{eq:Ia}, then
	the matrix
	\begin{equation}\label{eq:M0}
	M_0 := {\rm Diag}({\rm sign}(a)) P_0  {\rm Diag}({\rm sign}(a))
	\end{equation}
	is an element in the set $\partial_{\rm HS} {\rm Prox}_{\rho\|w\circ\cdot\|_1^2}(a)$. Moreover, if we define $\xi\in \mathbb{R}^{n}$ with $\xi_i = 0$ if $i\in I(|a|)$, $\xi_i = 1$ otherwise, and $\tilde{w} :=({\rm sign}(a)\circ \xi)\circ w$, then $M_0$ defined in \eqref{eq:M0} can be computed as
	\begin{equation*}
	M_0 = {\rm Diag}(\xi)-\frac{2\rho}{1+2\rho (\tilde{w}^T\tilde{w})} \tilde{w}\tilde{w}^T.
	\end{equation*}
\end{proposition}
\begin{proof}
	The first part of the proposition follows immediately from the definition of $\partial_{\rm HS} {\rm Prox}_{\rho\|w\circ\cdot\|_1^2}(a)$. Similarly to the proof in Proposition \ref{prop: property_HS_prox}, we can obtain that
	\begin{align*}
	P_0 = ({\rm Diag}(\xi) Q {\rm Diag}(\xi))^{\dagger}.
	\end{align*}
	Denote $\hat{w}=\xi \circ w$. Since $Q = I_n +  2\rho ww^T \in \mathbb{R}^{n\times n}$, it holds that
	\begin{align*}
	P_0 &= ({\rm Diag}(\xi) Q {\rm Diag}(\xi))^{\dagger}
	= ({\rm Diag}(\xi) + 2\rho \hat{w}\hat{w}^T)^{\dagger}={\rm Diag}(\xi) -
	\frac{2\rho}{1+2\rho (\hat{w}^T\hat{w})} \hat{w}\hat{w}^T.
	\end{align*}
	Next we show that ${\rm sign}(a)\circ{\rm sign}(a)\circ\xi = \xi$. First, we note that $a_i=0$ implies that $(x(|a|))_i=0$, and hence $\xi_i=0$. Thus if $a_i =0$, then ${\rm sign}(a_i)^2 \xi_i = 0 = \xi_i$. For the case when $a_i\not=0$, we clearly have ${\rm sign}(a_i)^2 \xi_i = \xi_i$. Similarly, we could prove that $\hat{w}^T\hat{w}=\tilde{w}^T\tilde{w}$. Now it is easy to see that
	\begin{align*}
	M_0 &=   {\rm Diag}({\rm sign}(a)) \Big({\rm Diag}(\xi)-
	\frac{2\rho}{1+2\rho (\hat{w}^T\hat{w})} \hat{w}\hat{w}^T\Big)  {\rm Diag}({\rm sign}(a)) \\
	&= {\rm Diag}(\xi) -
	\frac{2\rho}{1+2\rho (\tilde{w}^T\tilde{w})} \tilde{w}\tilde{w}^T.
	\end{align*}
	This completes the proof.
\end{proof}

\subsection{The proximal mapping of $\Delta^{\mathcal{G},w}(\cdot)$ and its generalized Jacobian}
\label{subsec: jacobian}
Based on the equation \eqref{eq: reformulation_exclusive} and the discussions in the previous subsection, we summarize the following proposition, which gives the proximal mapping of $\Delta^{\mathcal{G},w}(\cdot)$ and its corresponding generalized Jacobian.
\begin{proposition}\label{prop: proxmapping_and_jacobian}
	Given $\nu>0$, and $p(\cdot)=\Delta^{\mathcal{G},w}(\cdot)$. The following statements hold.
	
	\noindent (1) The proximal mapping ${\rm Prox}_{\nu p}(\cdot)$ can be computed as
	\begin{align*}
	{\rm Prox}_{\nu p}(x)&={\cal P}^T
	\underset{y\in \mathbb{R}^n}{\arg\min}\  \Big\{
	\frac{1}{2}\|y-{\cal P}x\|^2+\nu \sum_{j=1}^l \|({\cal P} w)^{(j)} \circ y^{(j)}\|_1^2\Big\}\\
	&={\cal P}^T \big[{\rm Prox}_{\nu\|({\cal P} w)^{(1)}\circ\cdot\|_1^2}(({\cal P} x)^{(1)});\cdots;{\rm Prox}_{\nu\|({\cal P} w)^{(l)}\circ\cdot\|_1^2}(({\cal P} x)^{(l)})\big],
	\end{align*}
	where ${\rm Prox}_{\nu\|({\cal P} w)^{(j)}\circ\cdot\|_1^2}(\cdot)$, for each $j=1,\cdots,l$, is defined in Proposition \ref{prop:proxmapping}.
	
	\noindent(2) Define the multifunction $\partial_{\rm HS}{\rm Prox}_{\nu p}: \mathbb{R}^n\rightrightarrows \mathbb{R}^{n\times n}$ as
	\begin{align*}
	\partial_{\rm HS}{\rm Prox}_{\nu p}(x)=\left\{ {\cal P}^T {\rm Diag}(M_1,\cdots,M_l){\cal P}\mid M_j\in \partial_{\rm HS} {\rm Prox}_{\nu\|({\cal P} w)^{(j)}\circ\cdot\|_1^2}(({\cal P} x)^{(j)}), j = 1,\cdots,l	\right\},
	\end{align*}
	where $\partial_{\rm HS} {\rm Prox}_{\nu\|({\cal P} w)^{(j)}\circ\cdot\|_1^2}(\cdot)$, for each $j = 1,\cdots,l$, is defined in \eqref{eq: Jacobian_prox_12}. Then $\partial_{\rm HS}{\rm Prox}_{\nu p}(\cdot)$ can be regarded as the generalized Jacobian of ${\rm Prox}_{\nu p}(\cdot)$ satisfying the following properties.
	\begin{itemize}
		\item[(a)] $\partial_{\rm HS}{\rm Prox}_{\nu p}(\cdot)$ is a nonempty, compact valued and upper-semicontinuous multifunction;
		\item[(b)] for any $x\in \mathbb{R}^n$, the elements in $\partial_{\rm HS}{\rm Prox}_{\nu p}(x)$ are all symmetric and positive semidefinite;
		\item[(c)] ${\rm Prox}_{\nu p}(\cdot)$ is strongly semismooth with respect to $\partial_{\rm HS}{\rm Prox}_{\nu p}(\cdot)$.
	\end{itemize}
	In addition, we could construct a specific element in $\partial_{\rm HS}{\rm Prox}_{\nu p}(x)$ according to Proposition \ref{compute_M}.
\end{proposition}

\section{Numerical experiments}
\label{sec:numerical}
In this section, we perform numerical experiments to evaluate the performance of our proposed AS strategy with the PPDNA for solving exclusive lasso models. For simplicity, we take the weight vector $w$ to be all ones, and the vector $c$ to be zero. The exclusive lasso model can be described as
\begin{equation}\label{eq: exclusive-lasso}
\min_{x\in \mathbb{R}^n} \ \Big\{h(Ax) + \lambda \sum_{j=1}^l \| x_{g_j}\|_1^2\Big\}.
\end{equation}
All our computational results are obtained by running {\sc Matlab} on a windows workstation (12-core, Intel Xeon E5-2680 @ 2.50GHz, 128G RAM).

In the numerical experiments, we mainly focus on three aspects.
\begin{itemize}[noitemsep,topsep=0pt]
	\item[(1)] We compare our proposed PPDNA for solving \eqref{eq: exclusive-lasso} with three state-of-the-art first-order algorithms. The efficiency and scalability of the PPDNA indicates that it is the best choice for solving the exclusive lasso model with a
	given $\lambda >0$.
	\item[(2)] We compare the proposed AS strategy for generating a solution path for \eqref{eq: exclusive-lasso} with the commonly used warm-start approach. The numerical experiments demonstrate the extremely good performance of the AS strategy together with the PPDNA.
	\item[(3)] We apply the exclusive lasso model \eqref{eq: exclusive-lasso} to some real applications, including the index ETF (exchange traded fund) in finance, image and text classifications in multi-class classifications.
\end{itemize}

In our experiments, we measure the accuracy of the obtained solution by the relative KKT residual:
\begin{align*}
\eta_{\rm KKT} := \frac{\|x - {\rm Prox}_{\lambda p}(x - A^T\nabla h(Ax)\|}{1 + \|x\| + \|A^T\nabla h(Ax)\|}. 
\end{align*}
We terminate the tested algorithm when $\eta_{\rm KKT}\leq \varepsilon$, where $\varepsilon > 0$ is a given tolerance, which is set to be $10^{-6}$ by default.

\subsection{Comparison of the PPDNA with other algorithms}
\label{sec: numerical PPDNA}

In this subsection, we compare our proposed PPDNA for solving \eqref{eq: exclusive-lasso} with a given $\lambda>0$ to three state-of-the-art first-order algorithms: ILSA \cite{kong2014exclusive}, ADMM with step length $\kappa = 1.618$ \cite{fazel2013hankel} and APG with restart under the setting described in \cite{becker2011templates}. We take $\tau=1/\lambda_{\max}(AA^T)$ in the PPDNA. In the experiments, we also terminate PPDNA when it reaches the maximum iteration of $200$, and terminate ILSA, ADMM and APG when they reach the maximum iteration of $200,000$. In addition, we set the maximum computation time of each experiment as $1$ hour. To demonstrate the efficiency and scalability of the algorithms, we perform the time comparison on synthetic datasets over a range of scales.

\paragraph{The regularized linear regression problem with synthetic data.} In the model \eqref{eq: exclusive-lasso}, we take $h(y):=\sum_{i=1}^m (y_i-b_i)^2/2$, where $b\in \mathbb{R}^m$ is given. Motivated by \cite{campbell2017within}, we generate the synthetic data using the model $b = Ax ^* +\epsilon$, where $x ^*$ is the predefined true solution and $\epsilon \sim \mathcal{N}(0, I_m)$ is a random noise vector. Given the number of observations $m$, the number of groups $l$ and the number of features $p$ in each group, we generate each row of the matrix $A \in \mathbb{R}^{m \times lp}$ by independently sampling a vector from a multivariate normal distribution $\mathcal{N}(0, \Sigma)$, where $\Sigma$ is a Toeplitz covariance matrix with entries $\Sigma_{ij} = 0.9^{|i - j|}$ for features in the same group, and $\Sigma_{ij} = 0.3^{|i - j|}$ for features in different groups. For the ground-truth $x^{*}$, we randomly generate $10$ nonzero elements in each group with i.i.d values drawn from the uniform distribution on $[0,10]$.

We mainly focus on solving the exclusive lasso model in the high-dimensional settings. Hence, we fix $m$ to be $200$ and $l$ to be $20$, but vary the number of features $p$ in each group from $50$ to $1000$. That is, we vary the total number of features $n=lp$ from $1000$ to $20000$. To compare the robustness of different algorithms with respect to the parameter $\lambda$, we test all the algorithms under two different values of $\lambda$. More results for the PPDNA over a sequence of $\lambda$ can be found in the next subsection. The time comparison is shown in Figure \ref{fig: time-comparison}, which demonstrates the superior performance of PPDNA, especially for large-scale instances, comparing to ILSA, ADMM and APG. As one can observe, for the largest instance, PPDNA is at least a hundred times faster ADMM, which is the best performing first-order method.

\begin{figure}[H]
	\begin{center}
		\includegraphics[width = 0.8\columnwidth]{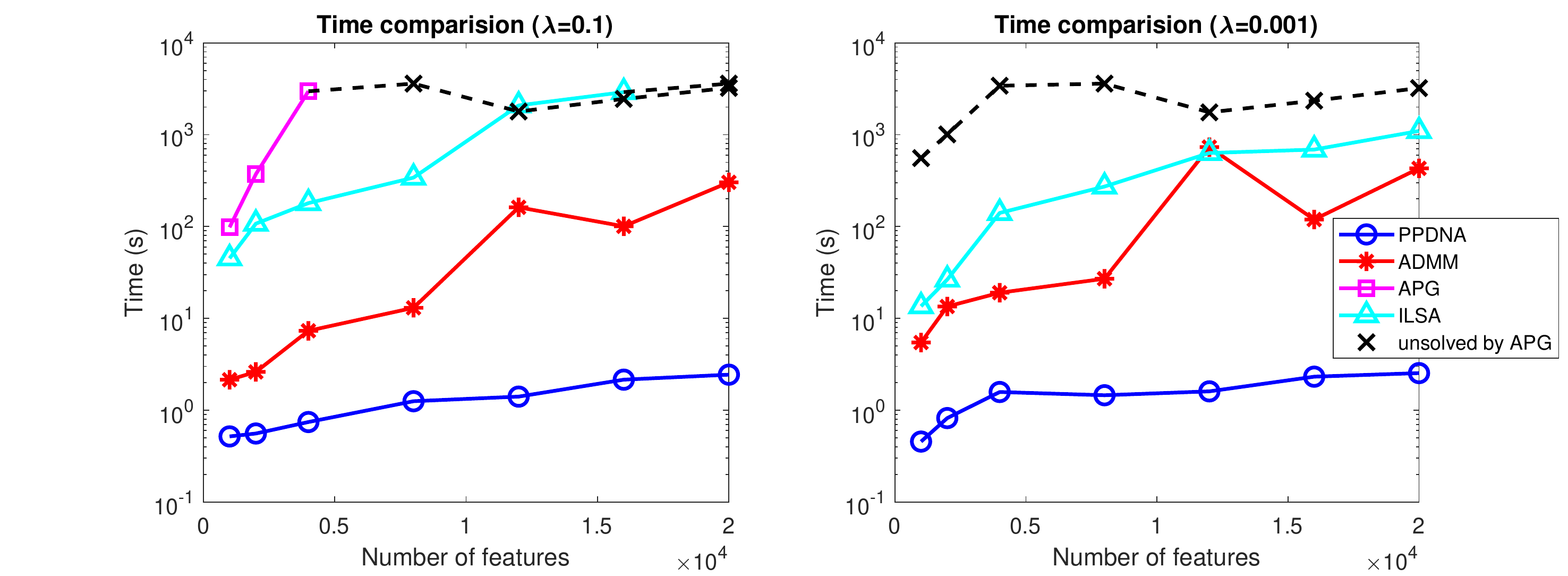}
		\setlength{\abovecaptionskip}{-2pt}
		\setlength{\belowcaptionskip}{-15pt}
		\caption{Time comparison among PPDNA, ILSA, ADMM and APG for linear regression on synthetic datasets. }
		\label{fig: time-comparison}
	\end{center}
\end{figure}

More numerical results on higher-dimensional cases (larger $m$ and larger $p$) are shown in Table \ref{tab: ls}. As one can see from Figure \ref{fig: time-comparison}, APG and ILSA are not efficient enough to solve large-scale instances, thus we only compare PPDNA with ADMM in these higher-dimensional cases. For the largest two instances in Table \ref{tab: ls}, PPDNA is able to solve the problems within one minute, whereas ADMM fails to solve them even after $1$ hour.

\begin{table}[H]\fontsize{8pt}{11pt}\selectfont
	\setlength{\abovecaptionskip}{0pt}
	\setlength{\belowcaptionskip}{0pt}
	\caption{Comparison between PPDNA and ADMM for linear regression on synthetic datasets. In the table, ``23(106)" means ``PPDNA iterations (total inner SSN iterations)". Time is in the format of (hours:minutes:seconds). Values in bold means that the algorithm fails to solve the instance to the required accuracy.}\label{tab: ls}
	\renewcommand\arraystretch{1.2}
	\centering
	\begin{tabular}{|c|c|c|c|c|}
		\hline
		& & iter &  $\eta_{\rm KKT}$ & time\\
		\hline
		Data $(m,l,p)$  & $\lambda$  & PPDNA $|$ ADMM & PPDNA $|$ ADMM & PPDNA $|$ ADMM \\
		\hline
		\multirow{2}*{\tabincell{c}{$(500,20,2000)$} }
		&1e-1 & 23(106) $|$ 23332  & 8.5e-7 $|$ 1.0e-6 & 0:00:10 $|$ 0:08:08  \\
		&1e-3 & 30(95) $|$ 167472  & 6.3e-7 $|$ \bf{1.5e-6}  & 0:00:11 $|$ 1:00:00 \\
		\hline
		\multirow{2}*{\tabincell{c}{$(500,20,3000)$} }
		&1e-1 & 23(97) $|$ 46226  & 3.9e-7 $|$ \bf{2.1e-6} & 0:00:11 $|$ 1:00:00  \\
		&1e-3 & 29(100) $|$ 50402  & 7.9e-7 $|$ \bf{9.0e-6} & 0:00:15 $|$ 1:00:01 \\
		\hline
		\multirow{2}*{\tabincell{c}{$(1000,20,2000)$} }
		&1e-1 & 21(132) $|$ 16208  & 5.0e-7 $|$ 1.0e-6 & 0:00:26 $|$ 0:09:03  \\
		&1e-3 & 28(160) $|$ 89242  & 7.8e-7 $|$ 1.0e-6 & 0:00:33 $|$ 0:50:41 \\
		\hline
		\multirow{2}*{\tabincell{c}{$(1000,20,4000)$} }
		&1e-1 & 22(107) $|$ 15644  & 7.1e-7 $|$ \bf{1.2e-5} & 0:00:25 $|$ 1:00:00  \\
		&1e-3 & 29(126) $|$ 15680  & 9.7e-7 $|$ \bf{3.6e-3} & 0:00:39 $|$ 1:00:01 \\
		\hline
	\end{tabular}
\end{table}

\paragraph{The regularized logistic regression problem with synthetic data.}  To test the regularized logistic regression problem, we take $h(y) = \sum_{i=1}^m \log(1 + \exp(-b_i y_i))$ in \eqref{eq: exclusive-lasso}, where $b\in \{-1,1\}^m$ is given. We use the same synthetic datasets described in the previous part, except for letting $b_i=1$ if $ Ax ^* +\epsilon\geq 0$, and $-1$ otherwise. As one can see in the previous experiments, APG and ILSA are very time-consuming when solving large-scale exclusive lasso problems compared to PPDNA and ADMM. Thus for logistic regression problems, we only compare PPDNA with ADMM. The numerical results are shown in Table \ref{tab: logistic}. Again, we can observe the superior performance of PPDNA against ADMM, and the performance gap is especially wide when the parameter $\lambda=10^{-5}$. For example, PPDNA is at least $110$ times faster than ADMM in solving the instance $(500,20,5000)$ with $\lambda=10^{-5}$.

\begin{table}[H]\fontsize{8pt}{11pt}\selectfont
	\setlength{\abovecaptionskip}{0pt}
	\setlength{\belowcaptionskip}{0pt}
	\caption{Time comparison between PPDNA and ADMM for logistic regression on synthetic datasets.}\label{tab: logistic}
	\renewcommand\arraystretch{1.2}
	\centering
	\begin{tabular}{|c|c|c|c|c|}
		\hline
		& & iter &  $\eta_{\rm KKT}$ & time\\
		\hline
		Data $(m,l,p)$  & $\lambda$  & PPDNA $|$ ADMM  & PPDNA $|$ ADMM & PPDNA $|$ ADMM \\
		\hline
		\multirow{3}*{\tabincell{c}{$(500,20,5000)$} }
		&1e-1 & 12(45) $|$ 2167  & 2.6e-7 $|$ 1.0e-6 & 0:00:06 $|$ 0:04:45  \\
		&1e-3 & 37(51) $|$ 6187  & 2.1e-7 $|$ 1.0e-6 & 0:00:09 $|$ 0:10:17 \\
		&1e-5 & 67(68) $|$ 21584  & 8.9e-7 $|$ 1.0e-6 & 0:00:18 $|$ 0:33:54 \\
		\hline
		\multirow{3}*{\tabincell{c}{$(1000,20,8000)$} }
		&1e-1 & 13(50) $|$ 1947  & 9.1e-8 $|$ 1.0e-6 & 0:00:18 $|$ 0:14:02  \\
		&1e-3 & 57(69) $|$ 6991  & 9.2e-7 $|$ 1.0e-6 & 0:00:32 $|$ 0:39:01 \\
		&1e-5 & 89(90) $|$ 10519  & 9.7e-7 $|$ \bf{1.4e-5} & 0:01:01 $|$ 1:00:00 \\
		\hline
		\multirow{3}*{\tabincell{c}{$(2000,20,10000)$} }
		&1e-1 & 11(48) $|$ 1625  & 7.3e-7 $|$ 1.0e-6 & 0:00:46 $|$ 0:33:02  \\
		&1e-3 & 62(72) $|$ 3522  & 6.0e-7 $|$ \bf{5.8e-5} & 0:01:22 $|$ 1:00:05 \\
		&1e-5 & 79(80) $|$ 4415  & 9.9e-7 $|$ \bf{2.5e-4} & 0:02:26 $|$ 1:00:05 \\
		\hline
	\end{tabular}
\end{table}

\subsection{Solution path generation via the AS strategy}
In this section, we present the performance of our proposed AS strategy with PPDNA for generating a solution path of the exclusive lasso model in high dimensional cases. Although our AS strategy can adopt other algorithms for solving the involved reduced problems, we choose PPDNA since it outperforms other algorithms in solving the model for a fixed $\lambda$ as seen in the previous subsection. We take $h(y):=\sum_{i=1}^m (y_i-b_i)^2/2$ in this subsection.

\subsubsection{Efficient initialization via correlation test}
Unlike the lasso model, there may not exist a regularization parameter value $\lambda_{\rm{max}}$ such that the exclusive lasso model admits a trivial optimal solution of all zeros. Thus, we must solve the exclusive lasso model with an initial value $\lambda_0$ even if we choose $\lambda_0$ to be large. Of course, we can directly solve the exclusive lasso model with $\lambda_0$. However, this may be expensive for high dimensional problems even with the powerful numerical algorithm PPDNA. Here, we propose a strategy for solving the exclusive lasso model with a large $\lambda_0$ by taking the advantage of the sparsity of the solution as well as the AS strategy . The idea is actually quite intuitive, if we can guess a reasonable initial active feature index set, then we can adopt the sieving property of the AS strategy to correct the active feature index set adaptively.

Now, we discuss our strategy to guess the initial active feature index set $I^0(\lambda_0)$, which is similar to the idea of surely independent screening rule in \cite{fan2008sure}. First, we set the number of initial active features to be $\lceil \sqrt{n} \rceil$. Then we choose the initial active features based on the correlation test between each feature vector $a_i$ and the response vector $b$. That is, we compute $s_i := \frac{|\langle a_i, b \rangle|}{\|a_i\|\|b\|}$ for $i=1,\cdots,n$, and choose the initial guess of $I^0(\lambda_0)$ as
\begin{align*}
I^0(\lambda_0)=\left\{i\in \{1,\cdots,n\}: s_i\mbox{ is among the first } \lceil \sqrt{n} \rceil \mbox{ largest values in} \; s_1,\ldots,s_n \right\}.
\end{align*}
Then we can follow the procedure in {\bf Step 3} of the AS strategy to correct $I^0(\lambda_0)$ adaptively and compute an approximate optimal solution of the exclusive lasso model with $\lambda_0$.

To demonstrate the effectiveness of this strategy, we compare the running time of the PPDNA for solving the exclusive lasso model with a given $\lambda_0$ with and without guessing the initial active feature index set in Table \ref{tab: initialization_screening}. We can see that the proposed strategy substantially improves the efficiency of computing the solution of the exclusive lasso model with a given $\lambda_0$. In the subsequent experiments, we adopt this technique to initialize the AS strategy with the PPDNA.

\begin{table}[H]\fontsize{8pt}{11pt}\selectfont
	\setlength{\abovecaptionskip}{0pt}
	\setlength{\belowcaptionskip}{0pt}
	\caption{Time comparison of the PPDNA with and without the initial active feature guess}\label{tab: initialization_screening}
	\renewcommand\arraystretch{1.2}
	\centering
	\begin{tabular}{|c|c|c|c|c|}
		\hline
	    &	&  & $\eta_{\rm KKT}$ & time\\
		\hline
		Data $(m,l,p)$  & $\lambda_0$  & nnz$(x)$ & PPDNA with initial guess $|$ PPDNA & PPDNA with initial guess $|$ PPDNA \\
		\hline
		\multirow{2}*{\tabincell{c}{$(500, 20, 20000)$ } }
		& 10 & 29 & 1.52e-7 $|$ 7.28e-7 & 0:00:02 $|$  0:00:18 \\
		& 1 & 114 & 1.33e-7 $|$ 4.41e-7 & 0:00:02 $|$ 0:00:21  \\
		\hline
		\multirow{2}*{\tabincell{c}{$(1000, 20, 40000)$ } }
		& 10 & 33 & 2.80e-7 $|$ 4.03e-7 & 0:00:02 $|$ 0:00:21 \\
		& 1 & 145 & 6.94e-8 $|$ 3.04e-7 & 0:00:03 $|$ 0:01:25 \\
		\hline
		\multirow{2}*{\tabincell{c}{$(1500, 20, 60000)$ } }
		& 10 & 32 & 1.83e-7 $|$ 5.86e-7 & 0:00:02 $|$ 0:02:50  \\
		& 1 & 105 & 4.75e-8 $|$ 2.25e-7 & 0:00:05 $|$ 0:03:28 \\
		\hline
	\end{tabular}
\end{table}

\subsubsection{Solution path generation}
Now, we demonstrate the efficiency of our AS strategy with the PPDNA for generating solution paths of exclusive lasso models on synthetic datasets. We generate the random data in the same manner as described in section \ref{sec: numerical PPDNA} and test our algorithm on high dimensional datasets with the number of feature $n$ up to $1,200,000$. We choose the range of $\lambda$ from $1$ to $10^{-4}$ with $20$ equally divided grid points on the $\log_{10}$ scale. We summarize our numerical results in Figure \ref{fig: screening_synthetic}.

\begin{figure}[H]
	\begin{center}
		\hspace{-0.5cm}
		\includegraphics[width = 0.35\columnwidth]{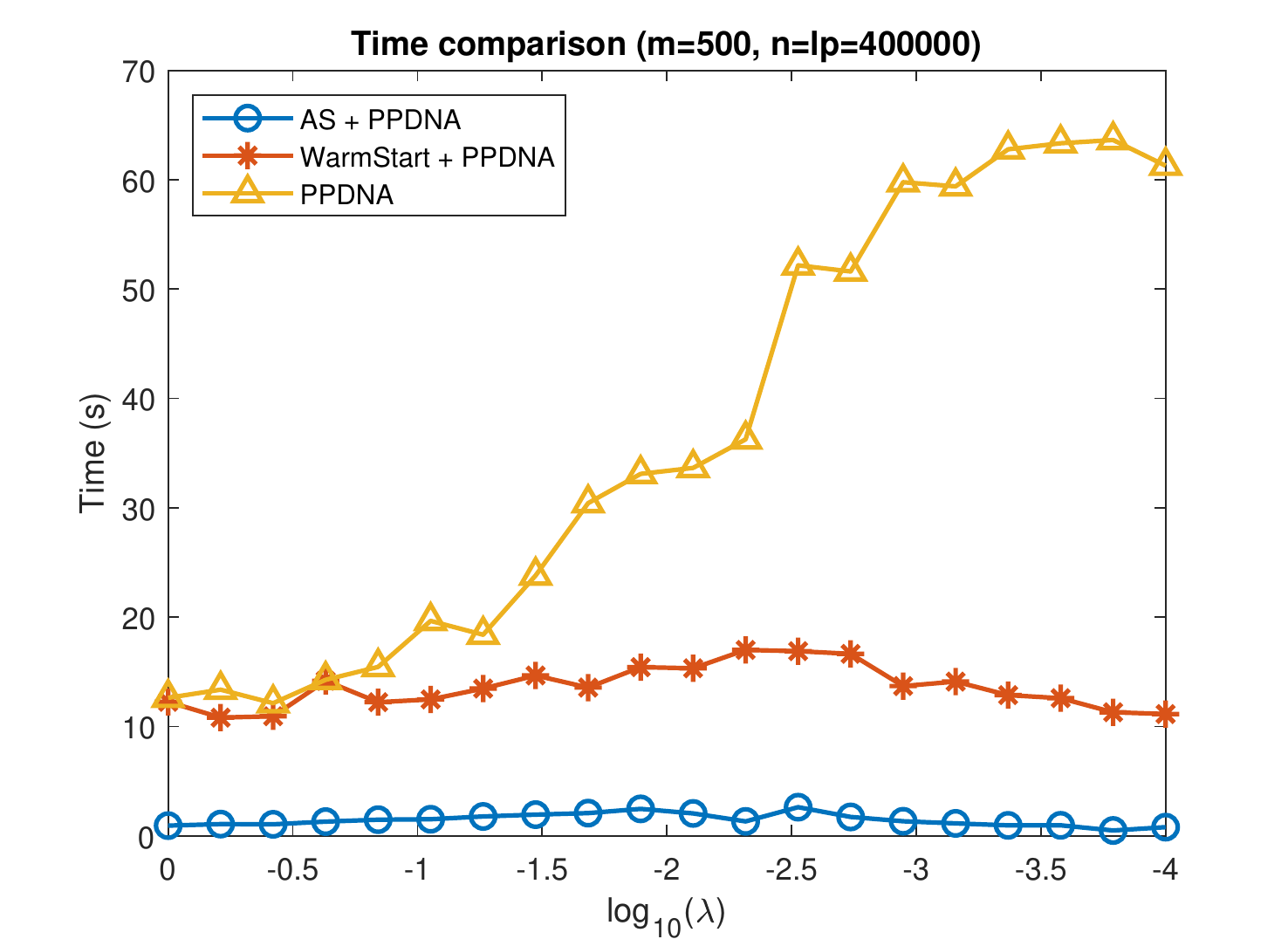}
		\hspace{-0.5cm}
		\includegraphics[width = 0.35\columnwidth]{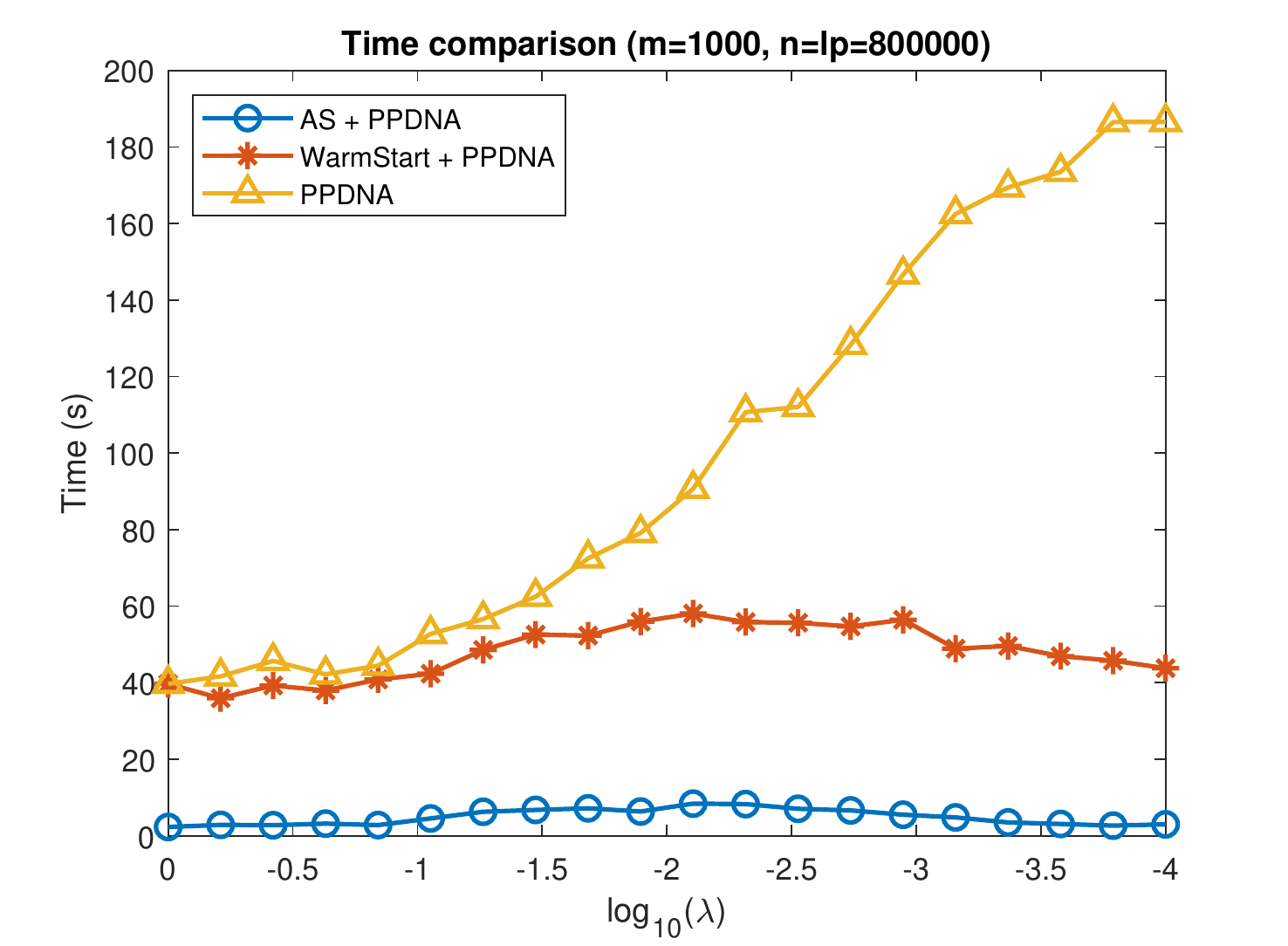}
		\hspace{-0.5cm}
		\includegraphics[width = 0.35\columnwidth]{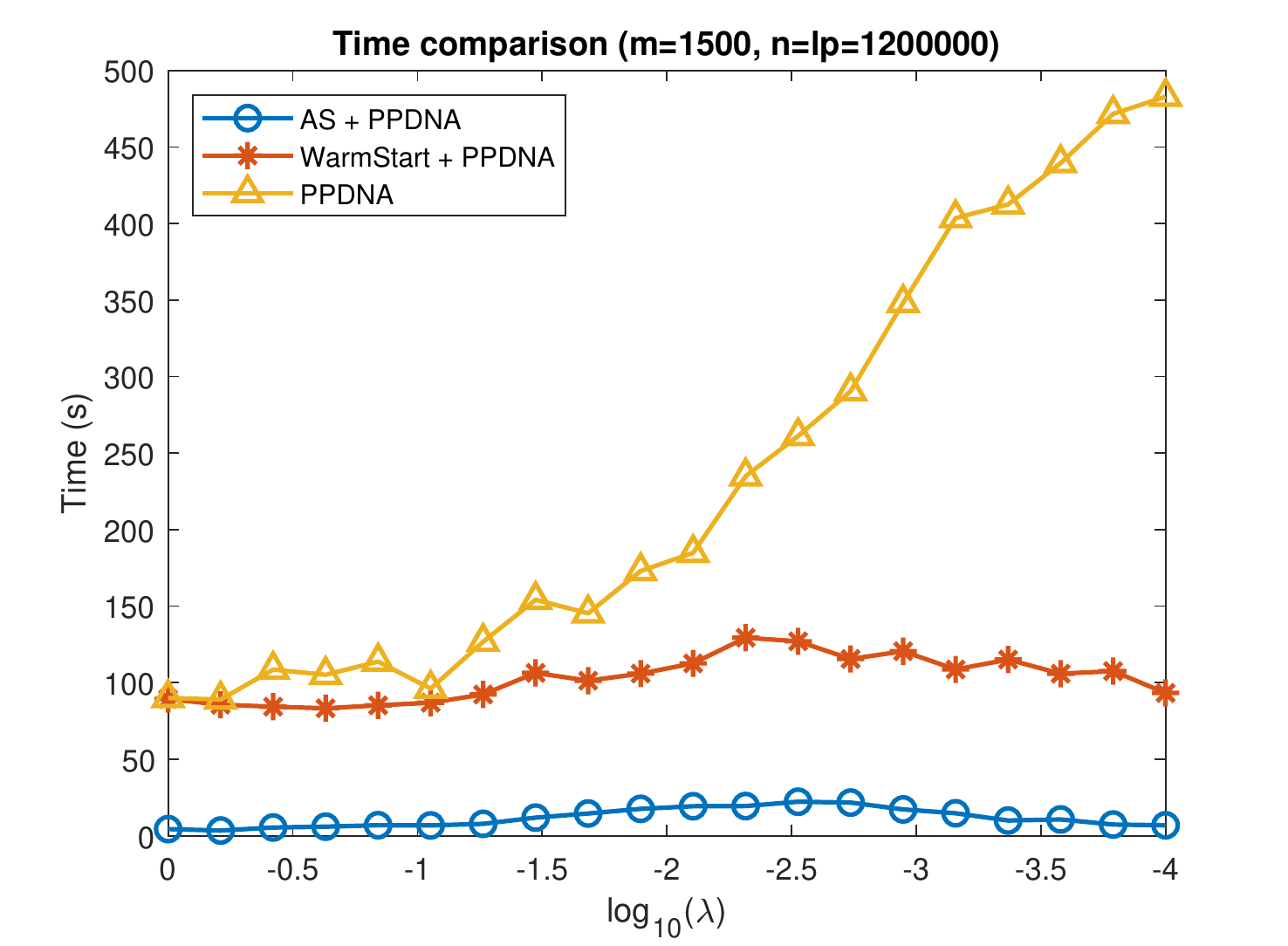}\\
		\hspace{-0.5cm}
		\includegraphics[width = 0.35\columnwidth]{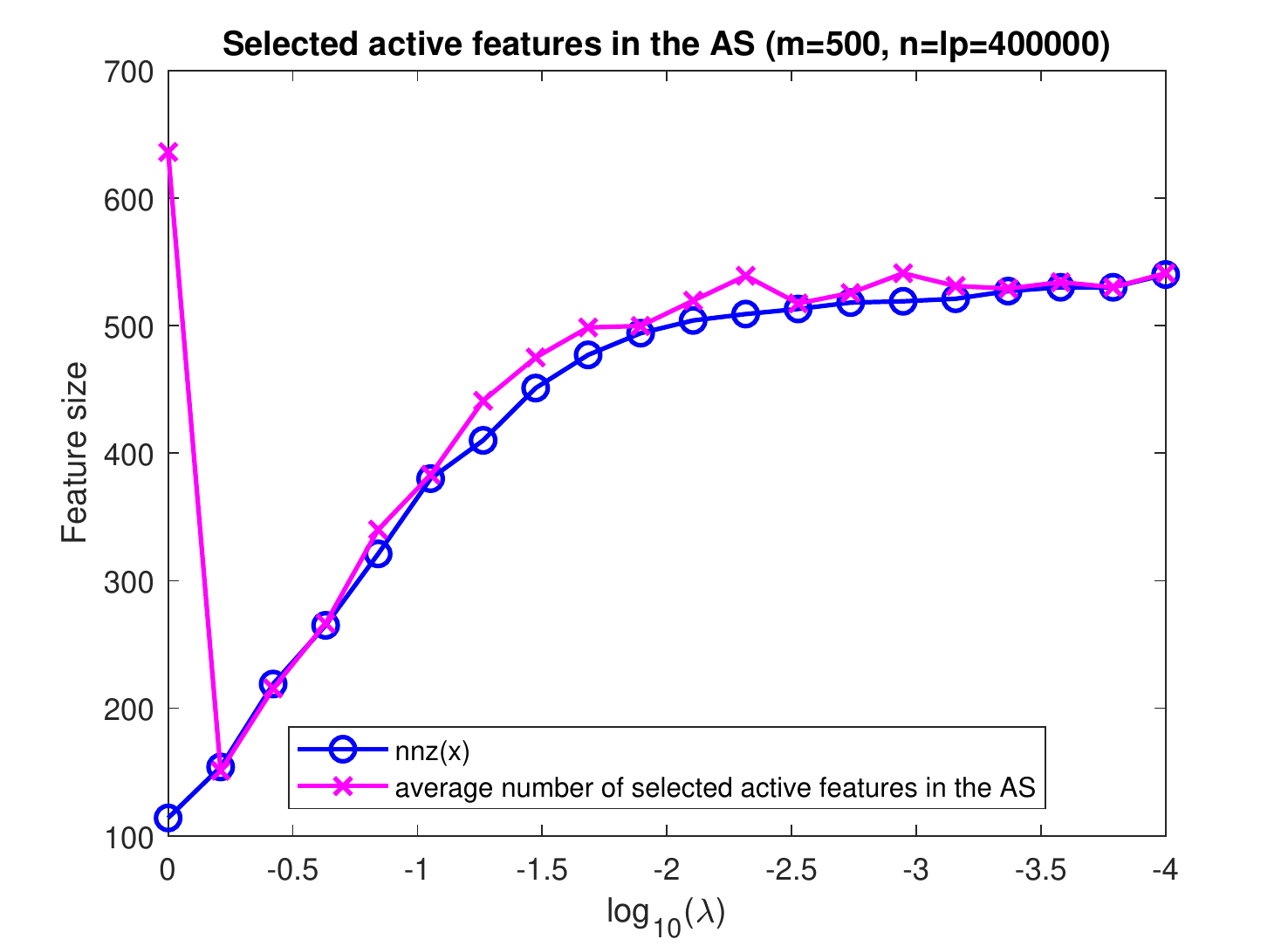}
		\hspace{-0.5cm}
		\includegraphics[width = 0.35\columnwidth]{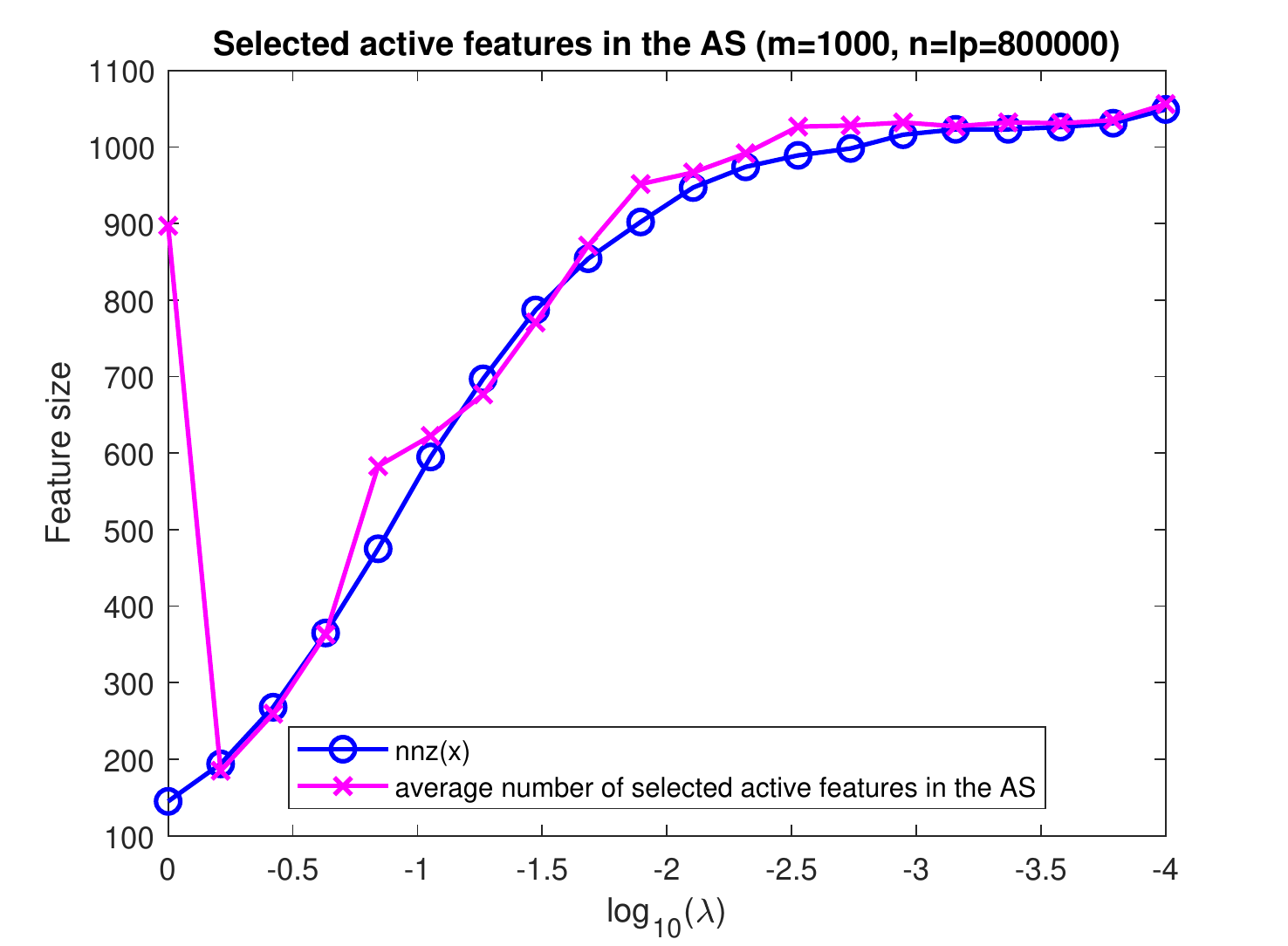}
		\hspace{-0.5cm}
		\includegraphics[width = 0.35\columnwidth]{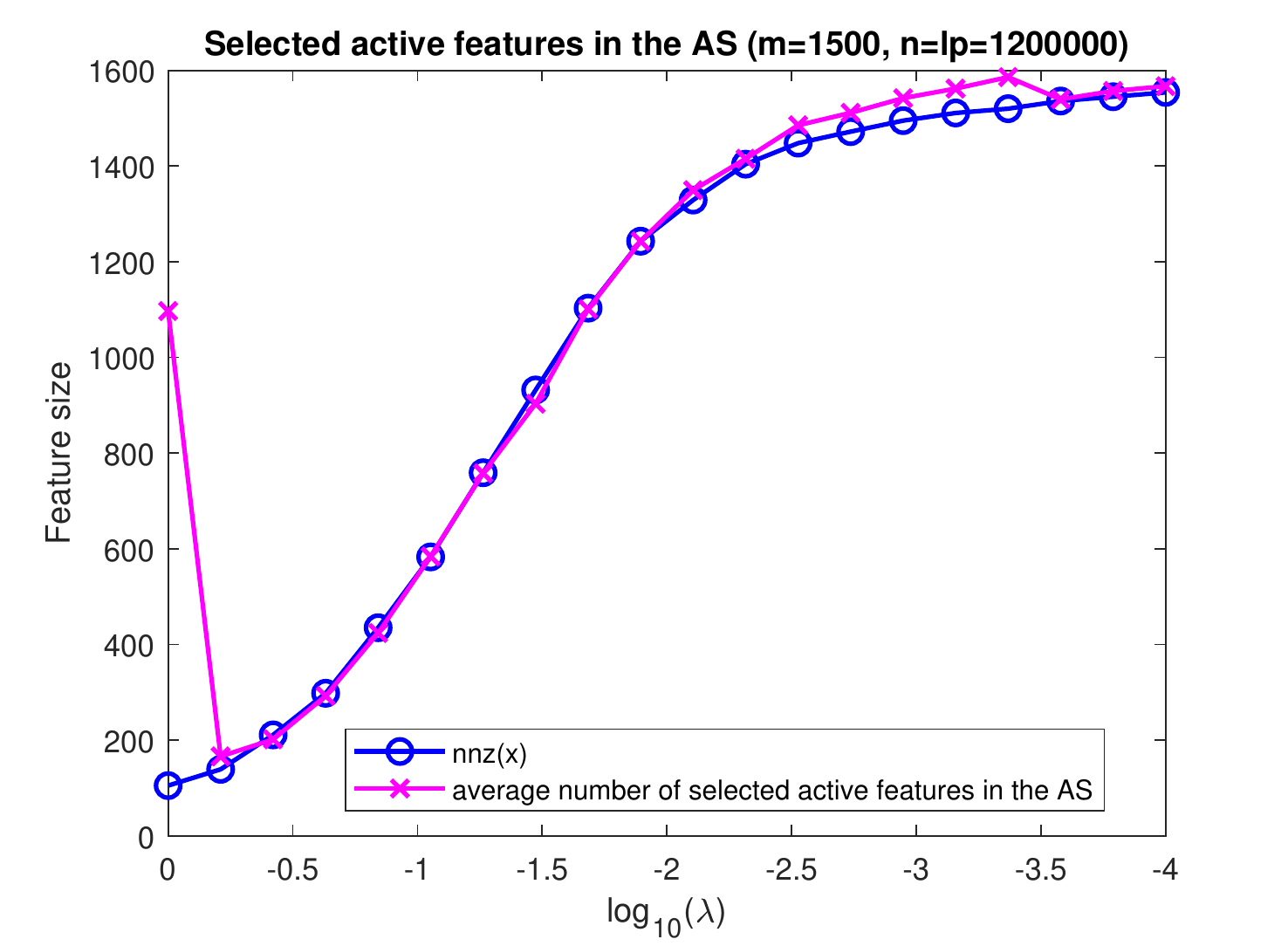}
		\setlength{\abovecaptionskip}{-10pt}
		\setlength{\belowcaptionskip}{-15pt}
		\caption{Numerical performance of the AS strategy with the PPDNA on synthetic datasets.}
		\label{fig: screening_synthetic}
	\end{center}
\end{figure}

In the first row of Figure \ref{fig: screening_synthetic}, we show the time comparison among three possible ways to generate solution paths of exclusive lasso models with different problem sizes. The results clearly show that our AS strategy can significantly improve the efficiency of the path generation. Compared to the obvious approach of warm-starting the PPDNA by using the optimal solution of the current problem as the initial iterate for solving the next problem with a new smaller parameter, our AS strategy can be at least 10 times more efficient for large instances. To uncover the reason behind, we plot the average problem size of the reduced problems in the AS strategy versus the number of nonzero entries of the optimal solution $x^*(\lambda)$ in the second row. It is surprising that the average reduced subproblem size in the AS strategy can nearly match the actual number of nonzeros in the optimal solution, except for the first problem with the parameter value $\lambda_0$.

In order to further demonstrate the power of the AS strategy, we show the time comparison of the AS+PPDNA with the oracle PPDNA in the left panel of Figure \ref{fig: screening_synthetic_ideal}. Here the oracle PPDNA means that we apply the PPDNA to the reduced problem based on the true active features, which are of course  impossible to know in practice, otherwise there is no need to do any feature selection. We could see that with the AS strategy, the running time is at most about 3 times longer than the oracle PPDNA, which shows that our AS is near-optimal, at least for the exclusive lasso models. In addition, the right panel of Figure \ref{fig: screening_synthetic_ideal} shows that we only need two or three rounds of sieving in the AS strategy to solve the problem for each $\lambda$, which shows that our sieving technique is quite effective in practice.

\begin{figure}[H]
	\begin{center}
		\includegraphics[width = 0.4\columnwidth]{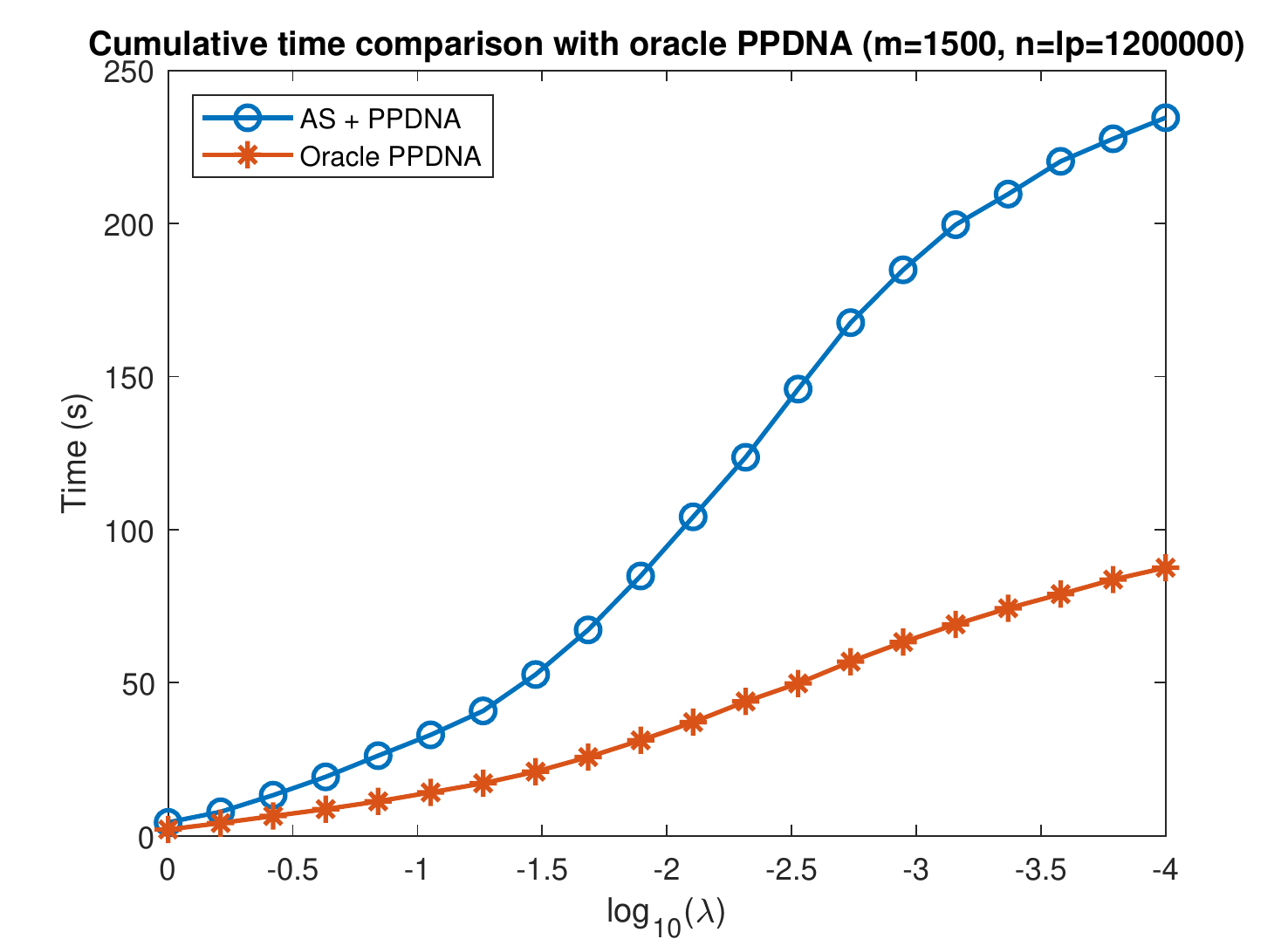}
		\includegraphics[width = 0.4\columnwidth]{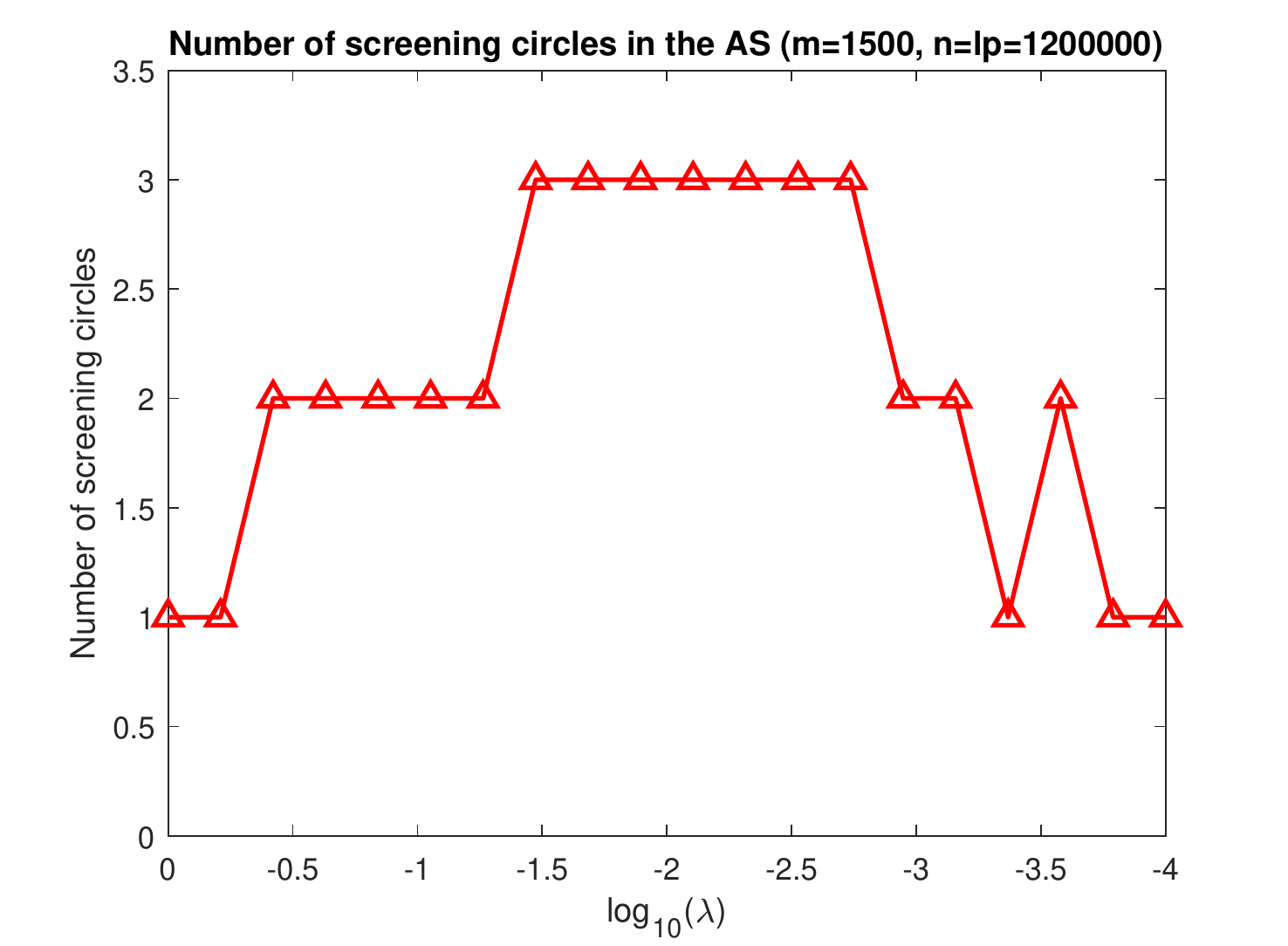}
		\setlength{\abovecaptionskip}{-2pt}
		\setlength{\belowcaptionskip}{-15pt}
		\caption{Performance profile for the case when $m=1500$, $n=1,200,000$.}
		\label{fig: screening_synthetic_ideal}
	\end{center}
\end{figure}

\subsection{Real applications}
In this subsection, we apply the exclusive lasso model to several real applications in finance and multi-class classifications.

\subsubsection{Index exchange-traded fund (index ETF)}
Consider the portfolio selection problem where a fund manager wants to select a small subset of stocks (to minimize transaction costs and business analyzes) to track the S\&P 500 index. In order to diversify the risks, the portfolio is required to span across all sectors. Such an application naturally leads us to consider the exclusive lasso model.

In our experiments, we download all the stock price data in the US market between 2018-01-01 and 2018-12-31 (251 trading days) from Yahoo finance\footnote{https://finance.yahoo.com}, and drop the stocks with more than 10\% of their price data being missed. We get 3074 stocks in our stock universe and handle the missing data via the common practice of forward interpolation. Then we denote the historical daily return matrix as $R \in \mathbb{R}^{250 \times 3074}$, and the daily return of the S\&P 500 index as $y \in \mathbb{R}^{250}$. Since there are 12 sectors in the US market (e.g., finance, healthcare), we have a natural group partition for our stock universe as $\mathcal{G} = \{g_1, g_2, \dots, g_{12}\}$, where $g_i$ is the index set for stocks in the $i$-th sector.

To test the performance of the exclusive lasso model in index tracking, we use the rolling window method to test the in-sample and out-of-sample performance of the model. We use the historical data in the last 90 trading days to estimate a portfolio vector via the model for the future 10 days. In each experiment, we scale the feature matrix $A$ and the response vector $b$ by $1/\sqrt{\|A\|_F}$, and select the parameter $\lambda$ in the range of $10^{-3}$ to $10^{-5}$ with $20$ equally divided grid points on the $\log_{10}$ scale, using 9-fold cross-validation. The in-sample and out-of-sample performance of the exclusive lasso model, the lasso model and the group lasso model is shown in Figure \ref{fig: partial-index-tracking}. The out-of-sample performance of the exclusive lasso model is visibly better than those corresponding to the lasso and group lasso models.

\begin{figure}[H]
	\begin{center}
		\includegraphics[width = 0.4\columnwidth]{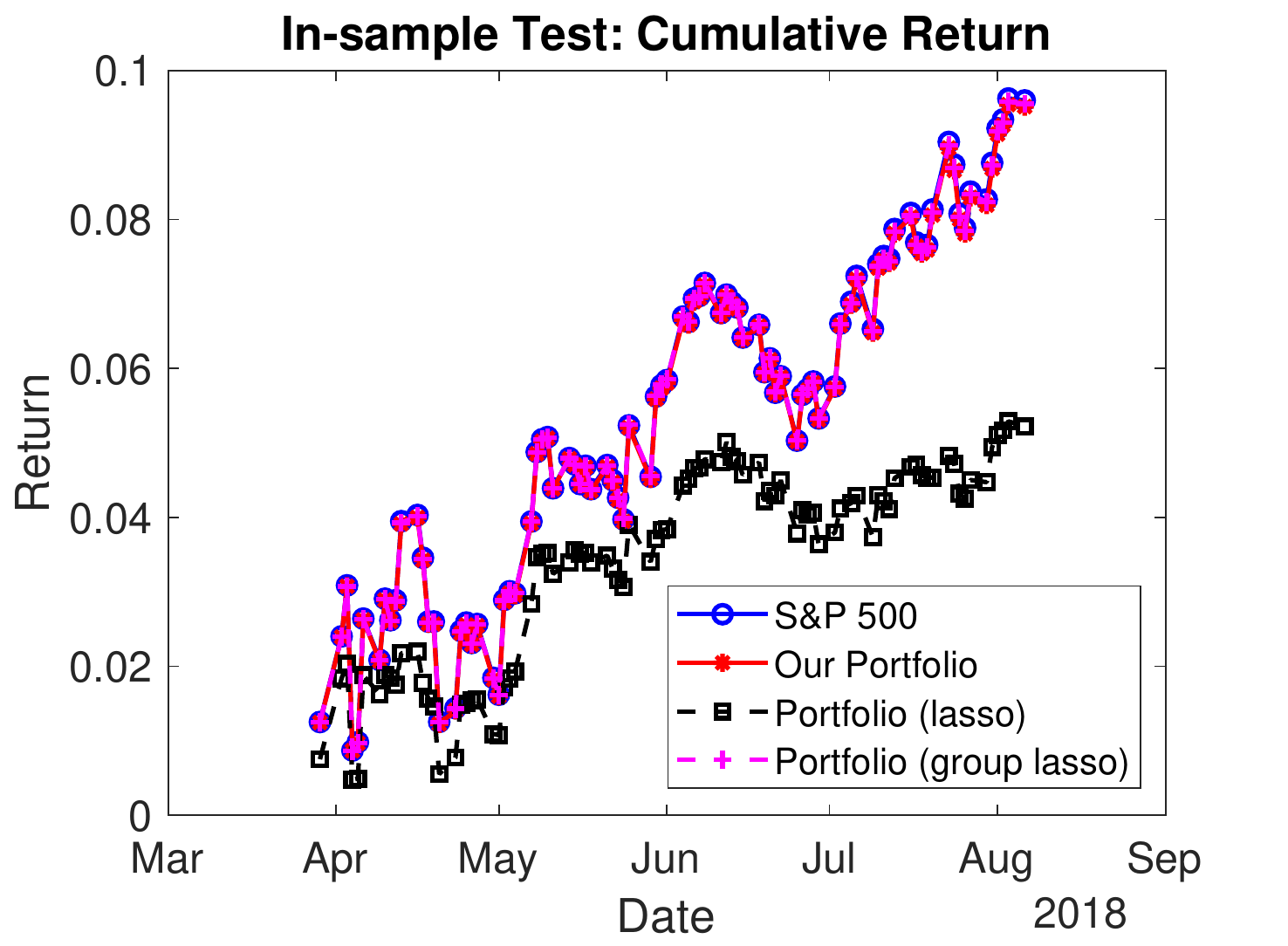}
		\quad
		\includegraphics[width = 0.4\columnwidth]{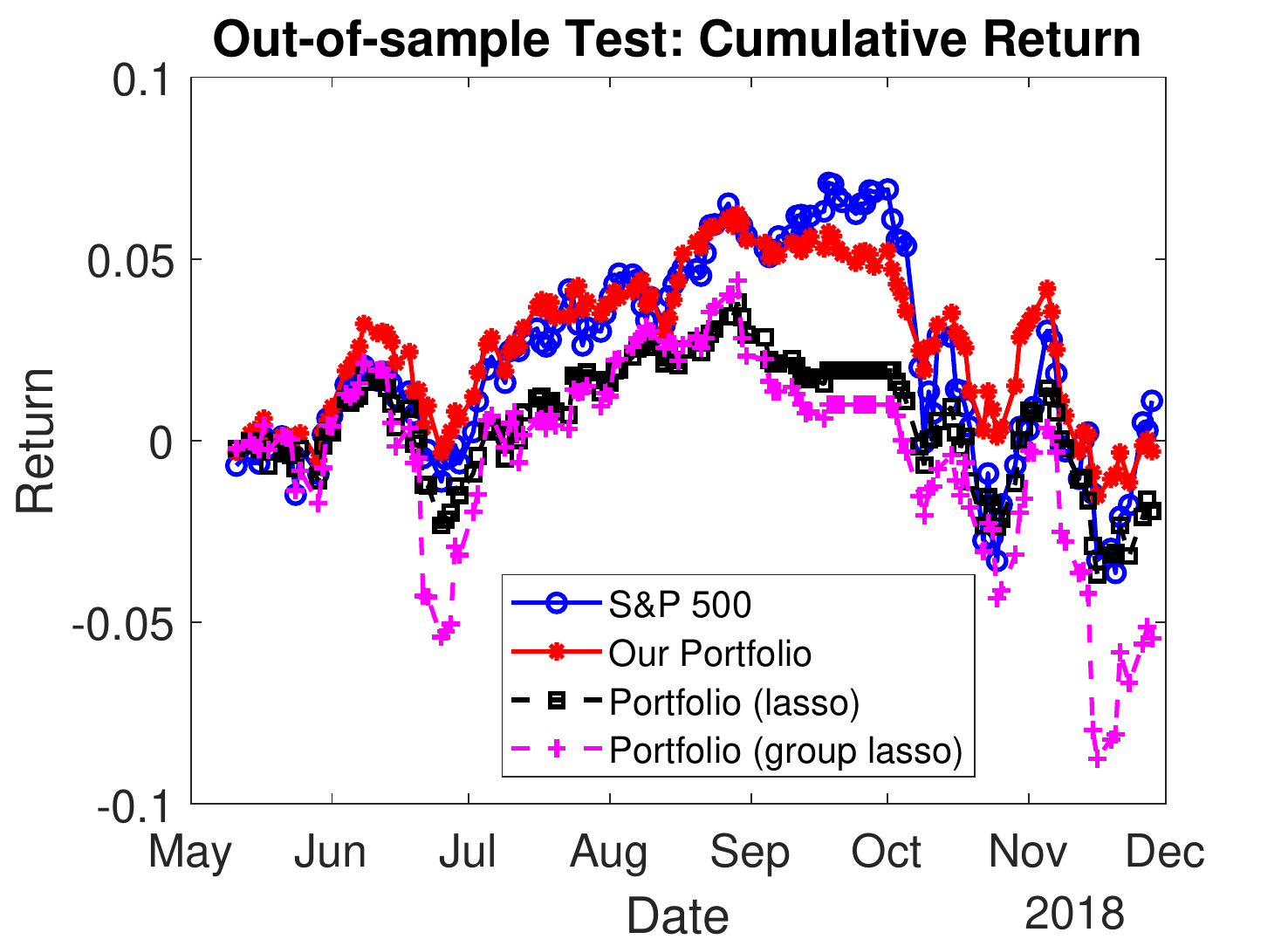}
		\setlength{\abovecaptionskip}{-2pt}
		\setlength{\belowcaptionskip}{-15pt}
		\caption{In-sample and out-of-sample performance of the exclusive lasso, the group lasso and the lasso model for index tracking of S\&P 500.}
		\label{fig: partial-index-tracking}
	\end{center}
\end{figure}

We plot the percentage of stocks from each sector in the portfolio obtained from the three tested models in Figure \ref{fig: pie-fig}. The result shows that the exclusive lasso model can select stocks from all the 12 sectors, but the lasso model selects stocks only from 10 sectors and the group lasso model selects stocks only from 7 sectors in the universe.

\begin{figure}[h]
	\addtocounter{figure}{1}
	\vspace{-0.4cm}
	\hspace{-1cm}
	\subfigure[the exclusive lasso model]{
		\label{fig_eg1}
		\includegraphics[width = 0.55\columnwidth]{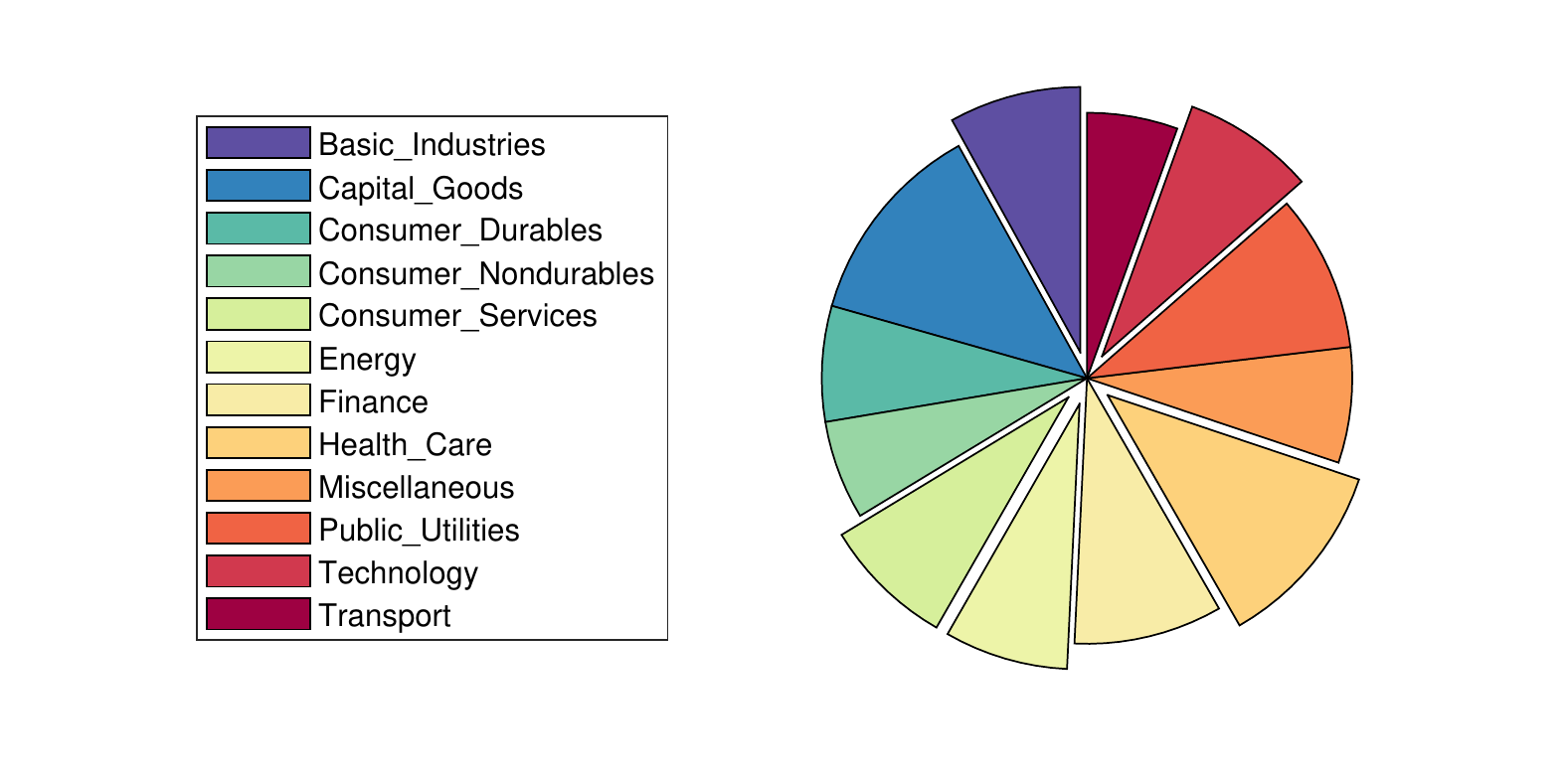}}
	\hspace{-1.2cm}
	\subfigure[the lasso model]{
		\label{fig_eg2}
		\includegraphics[width = 0.28\columnwidth]{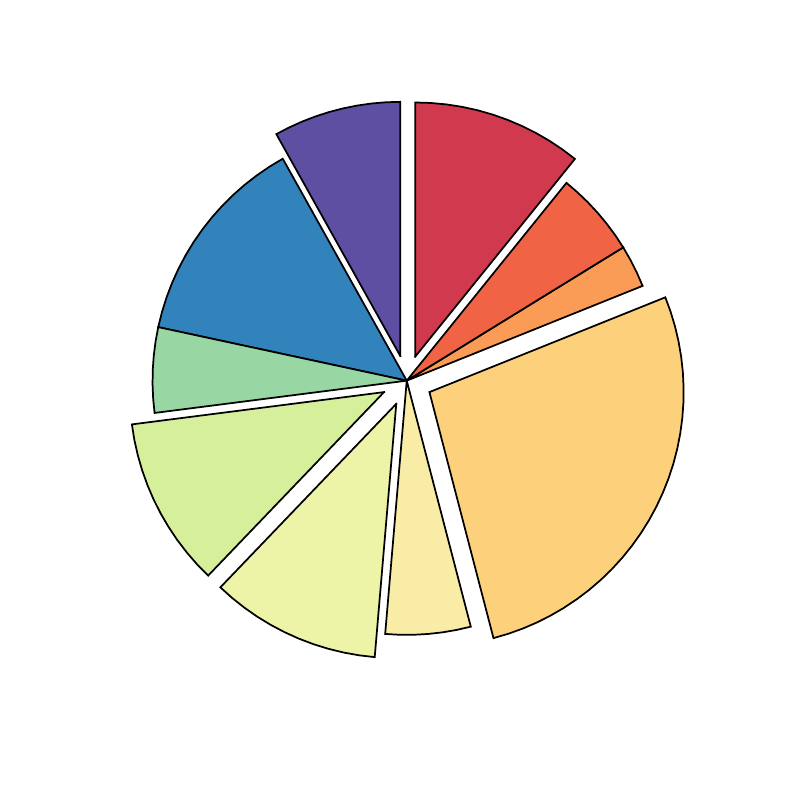}}
	\hspace{-0.8cm}
	\subfigure[the group lasso model]{
		\label{fig_eg3}
		\includegraphics[width = 0.28\columnwidth]{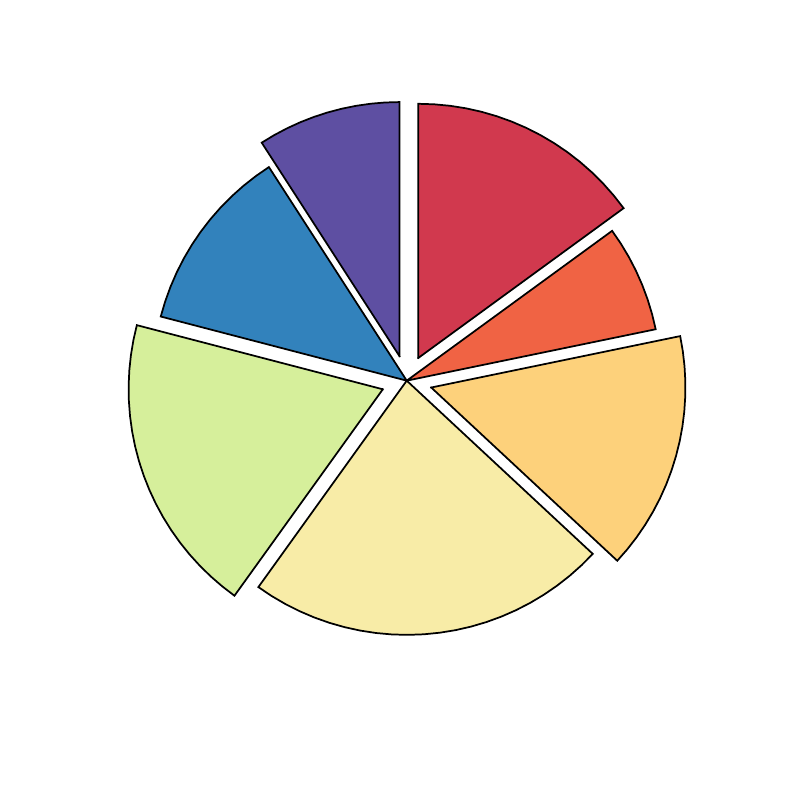}}
	\setlength{\abovecaptionskip}{0 pt}
	\setlength{\belowcaptionskip}{-15pt}
	\caption{Percentage of selected stocks by sectors.}
	\label{fig: pie-fig}
\end{figure}

\subsubsection{Image and text classifications}
We test the exclusive lasso model on multi-class classifications. For a given $k$-class classification dataset $\{(a_i, b_i)\}_{i=1}^N$, where $a_i \in \mathbb{R}^{p}$ is the feature vector and $b_i \in \mathbb{R}^k$ is the one-hot representation of the label, the exclusive lasso regression model for this problem \cite{zhou2010exclusive,kong2014exclusive,campbell2017within} is given by:
\begin{align}
\min_{X\in \mathbb{R}^{p\times k}}\ \Big\{\frac{1}{2}\|AX-b\|_F^2+\lambda \sum_{j=1}^p\|X_{j,:}\|_1^2
\Big\},\tag{M1}\label{eq: classification_M1}
\end{align}
where $A = [a_1, a_2, \dots, a_N]^T \in \mathbb{R}^{N \times p}$ and $b = [b_1, b_2, \dots, b_N]^T \in \mathbb{R}^{N \times k}$. The key motivation for considering this model is  to capture the negative correlation among the classes. However, the exclusive lasso regularizer may not exclude uninformative features if we penalize $X$ row-wise since it prefers to select at least one representative from each feature group. This phenomenon has also been discussed in a recent paper \cite{ming2019robust}. In our experiments, we consider the following model instead:
\begin{align}
\min_{X\in \mathbb{R}^{p\times k}}\ \Big\{\frac{1}{2}\|AX-Y\|_F^2+\lambda \sum_{j=1}^k\|X_{:,j}\|_1^2
\Big\}.\tag{M2}\label{eq: classification_M2}
\end{align}
The motivation for considering model \eqref{eq: classification_M2} is that we can do class-wise feature selections, since the informative features for different classes are usually not identical. Also, uninformative features will  automatically be excluded by the nature of class-wise feature selections. In order to show that the new model we suggest is meaningful, we first compare the model performance on two popular real datasets: MNIST \cite{lecun1998gradient} and 20 Newsgroups\footnote{http://qwone.com/$\sim$jason/20Newsgroups/}. We summarize the details of the datasets in Table \ref{tab:real_datasets}. Note that, after vectorization, the target problem size is actually $kN \times kp$.

\begin{table}[H]\fontsize{8pt}{11pt}\selectfont
	\setlength{\abovecaptionskip}{0pt}
	\setlength{\belowcaptionskip}{0pt}
	\caption{Details of real datasets. }\label{tab:real_datasets}
	\renewcommand\arraystretch{1.2}
	\centering
	\begin{tabular}{|c|c|c|c|c|}
		\hline
		Dataset & Num. of classes $k$ & Num. of samples $N$ & Num. of features $p$ &  Target problem size $(m,n)=(kN,kp)$\\
		\hline
		MNIST & 10 & 60000 & 784 & (600000, 7840)\\
		\hline
		20 Newsgroups & 20 & 11314 & 26214 & (226280, 524280)\\
		\hline
	\end{tabular}
\end{table}

We train \eqref{eq: classification_M1} and \eqref{eq: classification_M2} independently on the two datasets. As a prior knowledge, there are a certain percentage of features which are uninformative for these two datasets (e.g., background pixels for the MNIST dataset and some uninformative words for the 20 Newsgroups dataset). Thus in each experiment, we set a lower bound for the value of $\lambda$ such that no more than 90\% features are selected by the model. As a result, for the MNIST dataset, we train \eqref{eq: classification_M1} with $\lambda$ in the range from $10$ to $0.1$ and \eqref{eq: classification_M2} with $\lambda$ from $10$ to $10^{-3}$ with grid search and cross-validation. Similarly, for the 20 Newsgroup dataset, we train \eqref{eq: classification_M1} with $\lambda$ from $1$ to $10^{-3}$ and \eqref{eq: classification_M2} with $\lambda$ from $1$ to $10^{-6}$. We summarize the results in Table \ref{tab: numerical_results_realdata} and Figure \ref{fig: real_model_compare}. We can observe that, the classification accuracy of the two models are comparable, but model \eqref{eq: classification_M2} obviously performs better in terms of feature selections.
\begin{table}[H]\fontsize{7pt}{10pt}\selectfont
	\setlength{\abovecaptionskip}{0pt}
	\setlength{\belowcaptionskip}{0pt}
	\caption{Model comparison on real datasets.}
	\label{tab: numerical_results_realdata}
	\renewcommand\arraystretch{1.2}
	\centering
	\begin{tabular}{|c|c|c|c|c|c|c|}
		\hline
		Dataset & Model & $\lambda^*$  & total selected unique features & ${\rm nnz}(X)$ & training accuracy  & testing accuracy \\
		\hline
		\multirow{2}*{\tabincell{c}{ MNIST } }
		& (M1) & 1.0e-1 & 717 & 1922 & 84.01\% & 84.64\%  \\
		& (M2) & 1.0e-3 & 449 & 1818 & 84.03\% & 84.79\%  \\
		\hline
		\multirow{2}*{\tabincell{c}{ 20 Newsgroups } }
		& (M1) & 1.0e-3 & 25714 & 27537 & 88.15\% & 77.46\%  \\
		& (M2) & 1.0e-6 & 2789 & 5942 & 91.70\% & 79.14\%  \\
		\hline
	\end{tabular}
\end{table}

\begin{figure}[H]
	\vspace{-0.3cm}
	\flushright
	\includegraphics[width = 1\columnwidth]{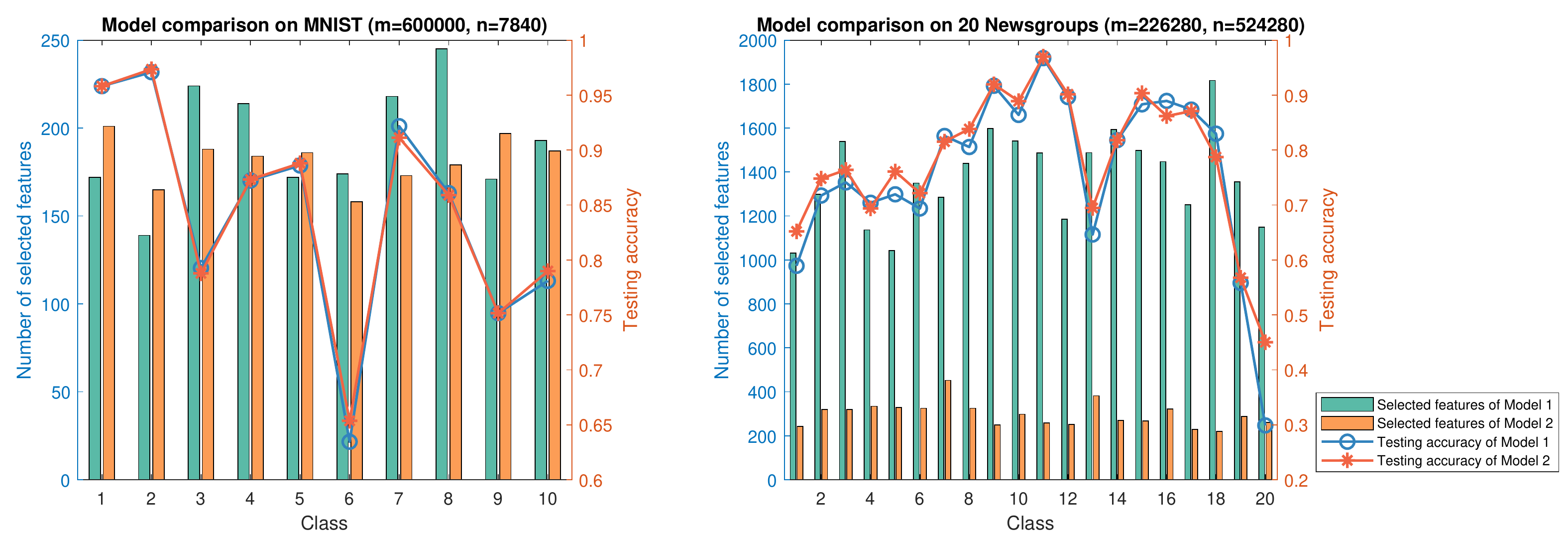}
	\setlength{\abovecaptionskip}{-15pt}
	\setlength{\belowcaptionskip}{-0pt}
	\caption{Model comparison between \eqref{eq: classification_M1} and \eqref{eq: classification_M2}. }
	\label{fig: real_model_compare}
\end{figure}

More importantly, as we can see in Table \ref{tab: numerical_results_realdata} and Figure \ref{fig: real_model_compare}, for the MNIST dataset, although the numbers of features in each class selected by two models are close, the total selected unique features of model \eqref{eq: classification_M2} is much less than that of model \eqref{eq: classification_M1}. This is because a group of important features which are selected by model \eqref{eq: classification_M2} are shared across different classes, which is consistent to our prior knowledge since almost all the targeted digits are located at the center of the images in the MNIST dataset. On the contrary, model \eqref{eq: classification_M1} selects 717 unique features out of the total 784 features, which means it selects a lot of uninformative features.

From now on, we focus on model \eqref{eq: classification_M2} and test the efficiency of our AS strategy with the PPDNA for solving the model with a sequence of hyper-parameters. For the two datasets, we generate solution paths for $\lambda$ over the range from 1 to $10^{-3}$ with 10 equally divided grid points on the $\log_{10}$ scale, and $10^{-3}$ to $10^{-6}$ with 10 equally divided grid points on the $\log_{10}$ scale, respectively. To fully demonstrate the power of AS+PPDNA approach for generating the solution paths, we compare its computation time with those of five other algorithms: AS+ADMM, PPDNA with warm-start, ADMM with warm-start, stand-alone PPDNA and stand-alone ADMM. We summarize the results in Figure \ref{fig: real_data}. We can observe that, AS+PPDNA is the best performer among all the six algorithms for generating solution paths of the exclusive lasso model. Note that PPDNA  beats ADMM by a large ratio. But AS+ADMM fares much better than ADMM, and its performance is closer to that of AS+PPDNA. This result also demonstrates the power of the AS strategy for solving large scale sparse optimization problems.

\begin{figure}[H]
	\begin{center}
		\includegraphics[width = 0.4\columnwidth]{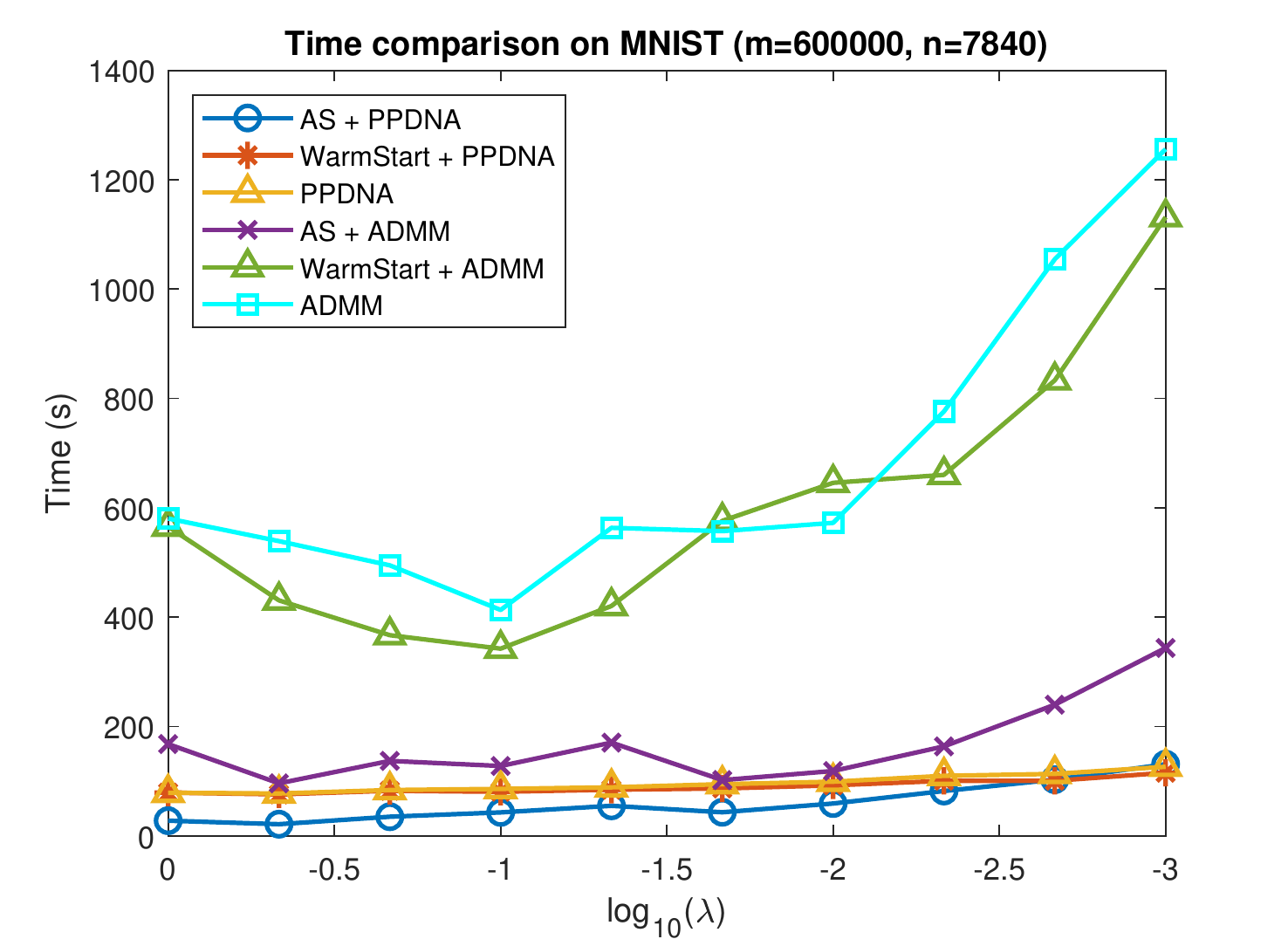}
		\quad
		\includegraphics[width = 0.4\columnwidth]{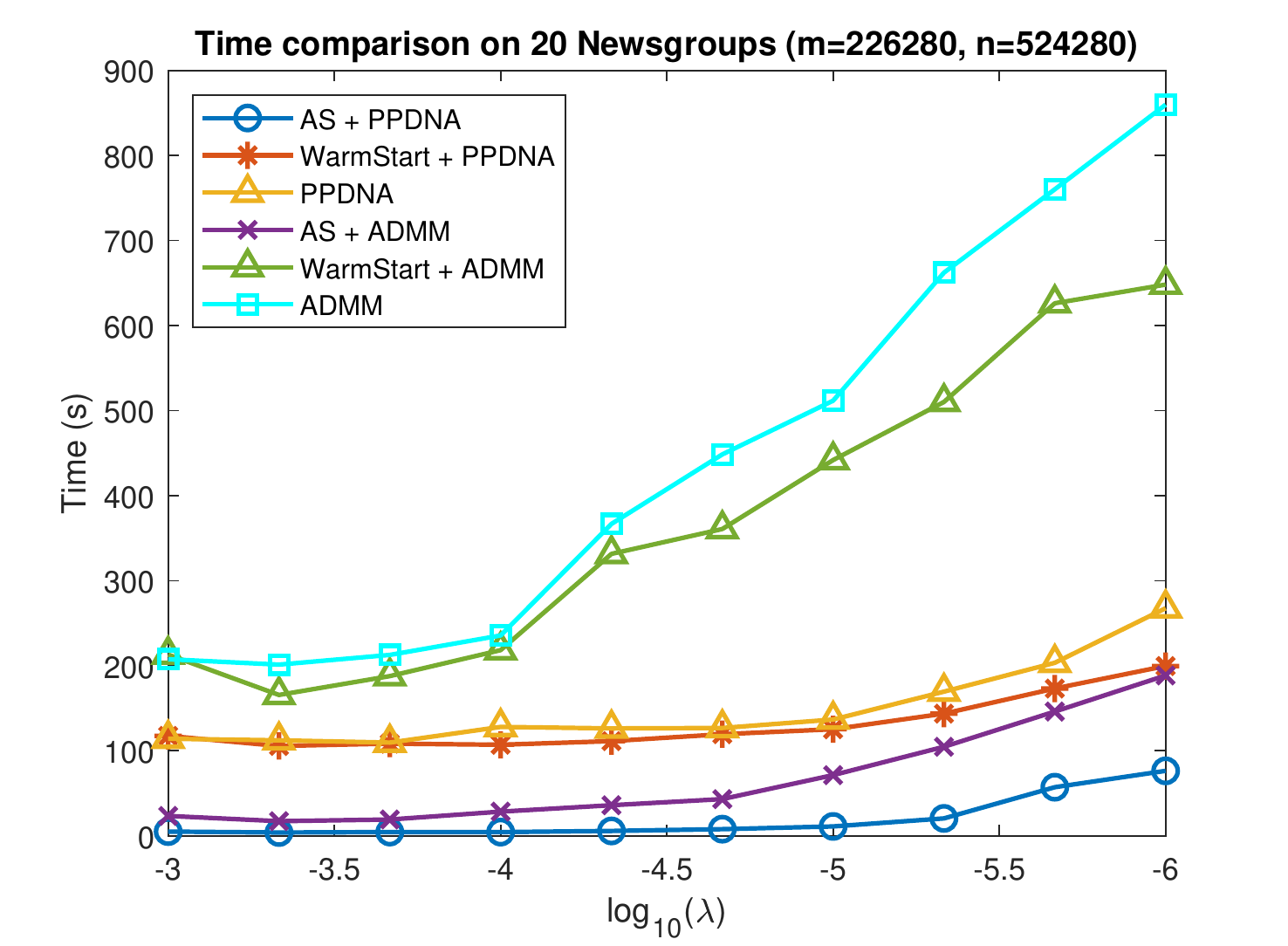}
		\setlength{\abovecaptionskip}{-2pt}
		\setlength{\belowcaptionskip}{-15pt}
		\caption{Time comparison on multi-class classifications.}
		\label{fig: real_data}
	\end{center}
\end{figure}

\section{Conclusion}
\label{sec:conclusion}

In this paper, we design an adaptive sieving strategy for generating solution paths of general machine learning models, including those with the exclusive lasso regularizer. In order to solve the reduced problems involved in the AS strategy for the exclusive lasso model, we design a highly efficient and scalable dual Newton method based proximal point algorithm, which is proved to converge superlinearly. As important ingredients, we systematically study the proximal mapping of the weighted exclusive lasso regularizer and its generalized Jacobian. Numerical experiments show that our AS strategy combined with the PPDNA is extremely efficient for generating solution paths of large-scale exclusive lasso models.

\appendix
\section*{Appendix}
\section{Numerical implementation of the SSN method}

In the SSN method presented in Algorithm \ref{alg:ssn}, the key step is to compute the Newton direction, in other words, to solve the linear system \eqref{eq: cg-system}. Here we 
present the numerical implementation details. In our implementation, we
explore the structured sparsity of the HS-Jacobian deeply which results in a highly efficient way to solve the linear system. 

Denote $\tilde{A}:=A {\cal P}^T$. Note that by the definition of 
the permutation matrix ${\cal P}$, $\tilde{A}$ could be obtained by permuting the columns in $A$ according to ${\cal P}$. Note that $\tilde{A}$ only needs to be computed once as a preprocessing step of the PPDNA algorithm since $\cal P$ is fully  determined by the fixed group information $\cal G$. Given $(\tilde{x},\tilde{u})\in\mathbb{R}^n\times \mathbb{R}^m$ and $\sigma,\tau>0$, we consider the following Newton system:
\begin{align}
\left( \frac{\sigma}{\tau}H  + \sigma \tilde{A} {\rm Diag}(M_1,\cdots,M_l)\tilde{A}^T\right) d = R,\label{eq: appendix_org}
\end{align}
where $R\in \mathbb{R}^m$ is a given vector, $H\in \nabla {\rm Prox}_{\sigma h/\tau}(A\tilde{x}+\frac{\sigma}{\tau}\tilde{u})$, $M_j\in \partial_{\rm HS} {\rm Prox}_{\sigma\lambda\|({\cal P} w)^{(j)}\circ\cdot\|_1^2}(({\cal P} \hat{x})^{(j)})$, $j=1,\cdots,l$, with $\hat{x}:=\tilde{x} + \sigma c - \sigma A^T\tilde{u}$. As shown in Proposition \ref{prop: jacobian_h}, $H$ is symmetric and positive definite. We denote the Cholesky decomposition of $H$ as $H=LL^T$, where $L$ is a nonsingular lower triangular matrix. Then we can reformulate the equation \eqref{eq: appendix_org} equivalently as
\begin{align*}
\left( \frac{\sigma}{\tau} I_m  + \sigma (L^{-1}\tilde{A})  {\rm Diag}(M_1,\cdots,M_l) (L^{-1}\tilde{A})^T\right) (L^T d) = L^{-1}R.
\end{align*}
Note that when we consider the linear regression or the logistic regression problems, the matrix $H$ is in fact a diagonal matrix, which means that we can compute $L$ and $L^{-1}$ with very low computational cost. For convenience, we write the linear system in a compact form as
\begin{align}
\left( I_m + \tau \hat{A} {\cal M} \hat{A}^T\right) \hat{d} = \hat{R},\label{eq: appendix_newton}
\end{align}
where $\hat{A}:= L^{-1}\tilde{A} \in \mathbb{R}^{m\times n}$, ${\cal M}:={\rm Diag}(M_1,\cdots,M_l)\in \mathbb{R}^{n\times n}$, $\hat{d} := L^T d \in \mathbb{R}^m$ and $\hat{R} := \frac{\tau}{\sigma}L^{-1}R \in \mathbb{R}^m$. Since $L^T$ is an upper triangular matrix, we can recover $d$ from $\hat{d}$ with the cost of $O(m^2)$. In the case of linear regression or the logistic regression problems, the cost of recovering $d$ from $\hat{d}$ is actually $O(m)$. Thus, we only need to focus on solving the linear system \eqref{eq: appendix_newton} for $\hat{d}$.

Based on the discussions in Proposition \ref{compute_M}, for each $j \in \{1,\cdots,l\}$, we could choose $M_j\in \partial_{\rm HS} {\rm Prox}_{\sigma\lambda\|({\cal P} w)^{(j)}\circ\cdot\|_1^2}(({\cal P} \hat{x})^{(j)})$ such that it has the following form:
\begin{align*}
M_j = {\rm Diag}(\xi_j)-\frac{2\sigma\lambda}{1+2\sigma\lambda (\tilde{w}_j^T\tilde{w}_j)} \tilde{w}_j\tilde{w}_j^T,
\end{align*}
where $\xi_j \in \mathbb{R}^{n_j-n_{j-1}}$ is a $0$-$1$ vector defined as $(\xi_j)_i = 0$ if $i\in I(|({\cal P} \hat{x})^{(j)}|)$, $(\xi_j)_i = 1$ otherwise, and $\tilde{w}_j=({\rm sign}(({\cal P} \hat{x})^{(j)})\circ \xi_j)\circ({\cal P} w)^{(j)}$, where $I(\cdot)$ is defined in \eqref{eq:Ia}.

We know that the costs of directly computing $\hat{A} {\cal M} \hat{A}^T$ and $\hat{A} {\cal M} \hat{A}^T\bar{d}$ for a given vector $\bar{d}\in\mathbb{R}^m$ are $O(m^2n)$ and $O(mn)$, respectively. This is computationally expensive when $m$ and $n$ are large. Next we will carefully explore the second-order sparsity of the underlying Jacobian which will substantially reduce the computational cost for solving the linear system \eqref{eq: appendix_newton}.

For each $j \in \{1,\cdots,l\}$, by taking advantage of the $0$-$1$ structure of $\xi_j$ and the definition of $\tilde{w}_j$, we have
\begin{align*}
M_j = {\rm Diag}(\xi_j)\left({\rm Diag}(\xi_j)-\frac{2\sigma\lambda}{1+2\sigma\lambda (\tilde{w}_j^T\tilde{w}_j)} \tilde{w}_j\tilde{w}_j^T\right){\rm Diag}(\xi_j)={\rm Diag}(\xi_j)M_j{\rm Diag}(\xi_j).
\end{align*}
Define $K_j:=\{k\mid (\xi_j)_k=1,k=1,\cdots,n_j-n_{j-1}\}$, $\xi:=[\xi_1;\cdots;\xi_l]\in \mathbb{R}^n$ and ${\cal K}:=\{k\mid (\xi)_k=1,k=1,\cdots,n\}$. It holds that
\begin{align}
\hat{A} {\cal M} \hat{A}^T = \hat{A}{\rm Diag}(\xi){\rm Diag}(M_1,\cdots,M_l){\rm Diag}(\xi) \hat{A}^T
=\hat{A}_{\cal K}{\rm Diag}(\hat{M}_1,\cdots,\hat{M}_l) \hat{A}_{\cal K}^T,\label{eq: appendix_AMAT}
\end{align}
where $\hat{A}_{\cal K}\in \mathbb{R}^{m\times |{\cal K}|}$ is the matrix consisting of the columns of $\hat{A}$ indexed by ${\cal K}$, and for each $j \in \{1,\cdots,l\}$, $\hat{M}_j\in \mathbb{R}^{| K_j| \times |K_j|}$ is defined as
\begin{align*}
\hat{M}_j = I_{|K_j|}-c_jv_jv_j^T,
\end{align*}
with $v_j:=(\tilde{w}_j)_{K_j}$, $c_j:=\frac{2\sigma\lambda}{1+2\sigma\lambda (\tilde{w}_j^T\tilde{w}_j)}$.

From the equation \eqref{eq: appendix_AMAT}, we can see that the costs of computing $\hat{A} {\cal M} \hat{A}^T$ and $\hat{A} {\cal M} \hat{A}^T \bar{d}$ for a given vector $\bar{d}\in \mathbb{R}^m$ reduce to $O(m^2|{\cal K}|)$ and $O(m|{\cal K}|)$, respectively. The reduction of the computation time is significant, since the sparsity of the solution induced by the exclusive lasso regularizer implies that $|{\cal K}|\ll n$. Note that when $m$ is moderate, we could use the Cholesky decomposition to solve the linear system \eqref{eq: appendix_newton} with the computational cost of $O(m^3+m^2|{\cal K}|)$. For the case when $|{\cal K}|\ll m$, we could use the Sherman-Morrison-Woodbury formula \cite{golub1996matrix} to further reduce the computational cost of solving \eqref{eq: appendix_newton}. To be specific, we have
\begin{align*}
\left(I_m+\hat{A}_{\cal K}{\rm Diag}(\hat{M}_1,\cdots,\hat{M}_l) \hat{A}_{\cal K}^T\right)^{-1}=I_m-\hat{A}_{\cal K}\left({\rm Diag}(\hat{M}_1^{-1},\cdots,\hat{M}_l^{-1})+\hat{A}_{\cal K}^T\hat{A}_{\cal K}
\right)^{-1}\hat{A}_{\cal K}^T,
\end{align*}
where for each $j \in \{1,\cdots,l\}$, 
\begin{align*}
\hat{M}_j^{-1} = I_{|K_j|}+\left( c_j^{-1}-v_j^Tv_j\right)^{-1} v_jv_j^T.
\end{align*}
Now, the cost of solving \eqref{eq: appendix_newton} is reduced to $O(|{\cal K}|^3+|{\cal K}|^2m)$. For the case when $m$ and $|{\cal K}|$ are both large, we could employ the conjugate gradient (CG) method to solve \eqref{eq: appendix_newton}, where the computational cost of each iteration of CG method is $O(m|{\cal K}|)$.

As one can see, in our implementation, we fully take advantage of the sparsity of the solution and the structure of the underlying Jacobian to highly reduce the computational cost of solving the Newton system \eqref{eq: cg-system}, which makes our SSN method efficient and robust for large scale problems.

\end{document}